\documentclass[a4paper,11pt]{article}

\usepackage{amsmath,amsfonts,amssymb,amsthm}
\usepackage{times}
\usepackage{graphicx}

\newcommand{\bR}{\mathbb{R}}

\newcommand{\bE}{\mathbb{E}}
\newcommand{\bP}{\mathbb{P}}

\newtheorem{theorem}{Theorem}

\newtheorem{definition}{Definition}

\newtheorem{proposition}[theorem]{Proposition}
 \rm

\DeclareMathOperator{\sgn}{sgn}

\numberwithin{equation}{section}

\begin{document}

\title{\bf Ruin Theory Problems in Simple SDE Models with Large Deviation Asymptotics}
\author{Efstathia Bougioukli and Michael A. Zazanis\thanks{Corresponding author: 
{\tt zazanis@aueb.gr}}\\ Department of Statistics\\
Athens University of Economics and Business, Athens, Greece 10434}
\date{}

\maketitle

\begin{abstract}
We examine hitting probability problems for Ornstein-Uhlenbeck (OU) 
processes and Geometric Brownian motions (GBM) with
respect to exponential boundaries related to problems 
arising in risk theory and asset and liability models in pension funds.
In Section 2 we consider the OU process described by the 
Stochastic Differential Equation (SDE) $dX_t = \mu X_t dt + \sigma dW_t$ 
with $X_0=x_0$ evolving between a lower and an upper deterministic 
exponential boundary. Both the finite horizon ``ruin probability'' problem 
and the corresponding infinite horizon problem is examined in the low noise 
case, using the Wentzell-Freidlin approach in order to obtain logarithmic 
asymptotics for the probability of hitting either the lower or the
upper boundary. The resulting variational problems are studied in detail. 
The exponential rate characterizing the ruin probability
and the ``path to ruin'' are obtained by their solution.
Logarithmic asymptotics for the meeting probability in a pair of OU 
processes with different positive drift coefficients,
driven by independent Brownian motions is also obtained using 
Wentzell-Freidlin techniques. The optimal paths followed by the two 
processes and the meeting time $T$ are determined by solving a variational 
problem with transversality conditions. In Section 3 a corresponding problem 
involving a Geometric Brownian motion is considered.
Since in this case, an exact, closed form solution is also available and we 
take advantage of this situation in order to explore
numerically the quality of the Large Deviations results obtained using the 
Wentzell-Freidlin approach. \\


\end{abstract}

\section{Introduction}

We examine simple linear Stochastic Differential Equations (SDE) describing Ornstein-Uhlenbeck (OU)
and Geometric Brownian motion (GBM) processes with positive drift and consider the ``ruin problem''
of hitting an upper or lower exponential boundary. This problem, is not analytically tractable in closed 
form for the OU process and we use the Wentzel-Freidlin approach in order to obtain Large Deviations 
estimates for the ruin probability. More specifically, suppose that the OU process describing the free 
reserves process $\{X_t\}$ is the solution of the SDE $dX_t = \mu X_t dt + \sigma dW_t$ with 
$X_0=x_0$ given, where $\mu>0$ and $\{W_t\}$ is standard Brownian motion.  
Suppose further that $V(t):=v_0e^{\beta t}$ and $U(t):= u_0 e^{\alpha t}$ are two exponential 
(deterministic) boundary curves, and that initially the free reserves lie between these values, i.e.\ 
$0<v_0<x_0<u_0$ and that $0< \beta < \mu < \alpha$. We examine both the finite horizon 
and infinite horizon ruin problem, i.e.\ the probability that the OU process remains between the two 
exponential bounds over a finite or an infinite time horizon.
These problems may of course be formulated in terms of a second order PDE with curved (exponential) 
boundaries in the plane and solved numerically. (An alternative approach, involving a time change 
argument is also discussed briefly.) The main thrust of the analysis however involves Large Deviations 
techniques and in particular the Wentzell-Freidlin approach in order to obtain logarithmic asymptotics for 
the probability of hitting either the lower or the upper boundary. These low-noise asymptotics are valid 
when the variance $\sigma$ is small and hence the event of hitting either boundary is rare. The 
exponential rate characterizing this probability is obtained by solving a variational problem which also 
gives the ``path to ruin''. We begin with a careful and detailed analysis of the finite horizon problem of 
hitting a lower boundary. The infinite horizon problem both for hitting the lower and the upper exponential 
boundary is treated using the transversality conditions approach of the calculus of variations. In addition, 
for the OU process with a more general linear drift resulting from the SDE 
$dX_t = (\mu X_t +r) dt + \sigma dW_t$, the probability of hitting an upper exponential boundary
$u_0 e^{\alpha t}$ is examined (with $0<\mu<\alpha$).

We also consider the problem of two independent OU processes, $\{X_t\}$, $\{Y_t\}$, with
initial values $x_0>y_0$ and average growth rates $\alpha$ and $\beta$ respectively with 
$\alpha > \beta$ so that,
in the absence of noise, it would hold that $X_t > Y_t$ for all $t>0$. We examine, again using
the Wentzell-Freidlin approach, the probability that the two processes meet. The optimal paths 
followed by the two processes and the meeting time $T$ is determined by solving a variational 
problem with trasversality conditions.

In section \ref{sec:GBM} a corresponding problem involving a Geometric Brownian motion 
described by the SDE $dX_t = \mu X_t dt + \sigma X_t dW_t$ with $X_0=x_0$ is examined, 
together with an upper and a lower exponential boundary. Again the Wentzell-Freidlin theory 
is used. In this case however, an exact solution is also possible, and therefore we are able to 
obtain an idea of the accuracy of the logarithmic asymptotics we propose. As expected, when 
the variance constant $\sigma$ becomes smaller, the quality of the approximation improves. 
The case of two correlated Geometric Brownian motions is also discussed. These models are 
inspired by the Gerber and Shiu (2003) model of assets and liabilities in pension funds.

Such models arise naturally when analyzing systems with compounding assets.
Consider the following collective risk model: Claims are i.i.d.\ random variables $\{Y_i\}$,
with distribution $F$ on $\mathbb{R}^+$, and they occur according to an independent Poisson
process with points $\{T_n\}$ and rate $\lambda$. Denote by 
$N_t:=\sum_{i=1}^\infty {\bf 1}(T_i \leq t)$ the corresponding counting process. Income 
from premiums comes at a constant rate $c$ and the initial value of the free reserves is $x_0$. 
We assume further that free reserves accrue interest at a fixed rate $\beta$. If we denote by 
$Z_t:= ct-\sum_{i=1}^{N_t}Y_i$, $t\geq 0$, the process describing net income (i.e.\ premium 
income minus liabilities due to claims), then the free reserves process is described by the 
SDE (Stochastic Differential Equation)
\begin{equation} \label{riskmodel}
dX_t \;=\;  \beta X_t dt + d Z_t, \;\;\; X_0 = x_0.
\end{equation}
Along the above lines, Harrison (1977) considered a generalization of the classical
model of collective risk theory in which the net income process of a firm, $\{Z_t\}$, has
stationary independent increments and finite variance and has paths that have finite 
variation w.p.1. Then the assets of the firm at time $t$, $X_t$, can be represented by a simple 
path-wise integral with respect to the income process $Z$ as
\begin{equation} \label{H1}
X_t = e^ {\beta t} x_0 + \int_0^t e^ { \beta (t-s)} dZ_s, \quad t \ge 0  ,
\end{equation}
where $x_0>0$ denotes the initial level of assets and $\beta>0$ the interest rate. 
In this setting the Riemann-Stieltjes definition can be used for the integral on the right side 
of (\ref{H1})  which is finite for all $t \ge 0$ and almost every sample path of $Z$. Thus the 
process $X$ is defined path-wise in terms of the income process $Z_t$.
A model with $Z_t$ being Brownian motion with drift would be natural as a diffusion 
approximation of such a model and this leads to the Ornstein-Uhlenbeck model we examine 
in detail in this paper.

Models with compounding assets occur naturally also in the study of pension funds. 
Gerber and Shiu (2003) have studied such models involving a pair of Geometric Brownian 
Motion processes with positive drift representing assets and liabilities over time and in this 
context ruin problems become relevant. With the notable exception of some Geometric 
Brownian Motion models, analytic solutions in closed form are not possible in general and 
thus we will study here ruin problems related to these models using Large Deviations 
techniques. For an overview of the vast subject of ruin probabilities in Risk Theory models 
see Asmussen and Albrecher (2010). See also Pham (2007) for an overview of large deviation 
techniques for estimating ruin probabilities.

The use of sample path large deviations for Stochastic Differential Equations with a low noise 
term was pioneered by Wentzell and Freidlin in the 1970's (see Wentzell and Freidlin, 2012, 
and the references therein) and Kifer (1974). In the 1980's this problem was studied in particular 
by Azencott (1980), Priouret (1982) and Azencott (1985). 
The type of problem we examine in this paper can be is related to the exit probability of a low 
noise diffusion from an unbounded domain, typically the positive quadrant of $\mathbb{R}^2$. 
Asymptotics for exit times and hitting probabilities of diffusions as well as applications to the
simulation of rare events have been studied in Cottrell, Fort, and Malgouyres (1985), Dupuis 
and Kushner (1986), Bobrovsky and Zeitouni (1992), Baldi (1995), Simonian (1995), Maier and
Stein (1997), and Freidlin, Koralov, and Wentzell (2017).

Exit problems for diffusions with low noise have also been studied extensively using analytic
techniques based on singular perturbation theory.  
Matkowski and Schuss (1977), Williams (1981) and the references therein.
See also Ludwig (1975) and Day (1984).

Baldi and Caramellino (2011) extend the Freidlin-Wentzell technique  to diffusions on the
positive half line with coefficients that are neither bounded nor Lipschitz conditions. Baldi (1995) 
and Robertson (2010) give applications of the Wentzell-Freidlin Large Deviation techniques to 
the simulation of diffusion processes. 
Low noise asymptotics using Wentzell-Freidlin techniques can also be used to study metastability
phenomena and {\em tipping events} in climatology, epidemiology, and ecological systems.  
For some recent such results see Hill, Zanetell and Gemmer (2022). 

Ruin problems related to Gaussian processes may also be analyzed using results from the vast 
literature on extreme values of Gaussian processes developed in the last four decades, obtaining
exact as opposed to logarithmic asymptotic results. We refer the reader to H\"usler (1990), 
Piterbarg (1996), H\"usler and Piterbarg (1999), Dieker (2005), and the references therein 
for an overview. 

Exact asymptotics have been obtained for various types of ruin problems 
involing integrated Gaussian processes in D\c{e}bicki (2002) and Kobelkov (2005) 
and also in He and Hu (2007) and D\c{e}bicki, Hashorva and Ji (2015) who
consider models that include force of interest. H\"usler and Piterbarg (2004) obtain exact aysmptotics
for the ruin probability of ``physical fractional Brownian motion'' which
is an integrated fractional Brownian motion with deterministic drift. H\"usler and Piterbarg (2008)
obtain the asymptotic distribution of the time to ruin for a class of Gaussian processes
with variance function which is regularly varying at infinity and drift $-ct^\beta$. 
More recently multidimensional problems considering simultaneous ruin in Gaussian
models have been considered as well. D\c{e}bicki, Hashorva, and 
Wang (2020) and Bisewski, D\c{e}bicki and Kriukov (2023) 
obtain exact asymptotics for simultaneous ruin in multivariate settings. 

The domain of applicability of these techniques is of course different from that of 
the Wentzell-Freidlin large deviation approach. 
The Wentzell-Freidlin approach used here gives only logarithmic asymptotics. However 
it is formulated as a variational problem whose solution gives the ``path to ruin'' and 
the time to ruin. Thus the solution gives valuable insights, useful in applications.
In Section \ref{sec:Gaussian} we discuss briefly how exact asymptotics may be
obtained using techniques from the theory of extremes of Gaussian processes in one
of the models examined in this paper.

\section{Low noise asymptotics for the Ornstein-Uhlenbeck process}

In this section we examine an Ornstein-Uhlenbeck (OU) process with positive infinitesimal drift
and consider the probability of hitting an upper or a lower exponential boundary. The problem is
approached using the Wentzell-Freidlin theory for obtaining logarithmic asymptotics both for the
finite and the infinite horizon problem. An OU process with an additional constant term in the
drift is also examined. Interestingly, depending on the value of the constant drift, the
variational problem from which the rate function is obtained, may or may not have a unique 
solution.

We begin by quoting a famous large deviations result for the Brownian motion which will be very 
important in the sequel.

\subsection{Large Deviation Results for the Paths of the Wiener Process}
Here we mention some results relating to Large Deviations. We refer the reader to Dembo and 
Zeitouni (2010) for further background.
Recall that a function $f$ is {\em lower semicontinuous} iff, for every sequence $\{x_n\}$ 
such that $\lim_{n\rightarrow \infty} x_n = x$, 
$\liminf_{n \rightarrow \infty} f(x_n) \geq f(x)$.
Let $(\mathcal{X},d)$ a metric space and $\mathcal{B}$ the Borel $\sigma$--field generated 
by the open sets of $\mathcal{X}$. If a function $I:\mathcal{X}\rightarrow [0,\infty]$ is 
lower semicontinuous then the level sets $\Psi_I(y):= \{x\in \mathcal{X}: I(x) \leq y\}$ are
{\em closed subsets of $\mathcal{X}$}. Suppose also that $\{\mu_n\}$ is a family 
of measures on $(\mathcal{X},\mathcal{B})$.

A lower semicontinuous function $I:\mathcal{X}\rightarrow [0,\infty]$ is called a
{\em good rate function} if the level sets $\Psi_I(y)$ are {\em compact} subsets of $\mathcal{X}$. 
The effective domain of the rate function $I$ is the subset of 
$\mathcal{X}$, $\mathcal{D}_I:= \{x: I(x)<\infty\}$ for which the rate function is finite.
As usual, for any $\Gamma \subset \mathcal{X}$, $\bar{\Gamma}$ denotes the {\em closure} 
and $\Gamma^o$ the {\em interior} of $\Gamma$. With the above definitions one may give the 
following precise statement of the Large Deviation Principle (LDP):
\begin{definition}
The family of measures on $\{\mu_\epsilon\}$ satisfies an LDP with rate function $I$ if for all
$\Gamma \in \mathcal{B}$,
\begin{equation} \label{LDP}
-\inf_{x\in \Gamma^o} I(x) \leq \liminf_{\epsilon \rightarrow 0} \epsilon \log \mu_\epsilon (\Gamma) \leq
\limsup_{\epsilon \rightarrow 0} \epsilon \log \mu_\epsilon (\Gamma) \leq - \inf_{x \in \bar{\Gamma}} I(x).
\end{equation}
\end{definition}

A fundamental result in sample path Large Deviations theory is the following theorem due to
Schilder (1966). Suppose that $\{W(t);t\in[0,1]\}$ is a Standard Brownian Motion in
$\mathbb{R}$ and let $\{W_\epsilon(t); t\in[0,1]\}$ denote a family of processes, indexed by
$\epsilon >0$ and defined via $W_\epsilon(t):= \sqrt{\epsilon}\, W(t)$. Here 
$\mathcal{X}=C_0[0,1]$, is the set of continuous real functions on $[0,1]$ which vanish at 0, 
equipped with the norm of uniform convergence. 
\begin{theorem}[Schilder] \label{th:Schilder} The family of measures $\{\mu_\epsilon\}$ 
induced on $C_0[0,1]$ by the family of processes $\{W_\epsilon(t);t\in[0,1]\}$ satisfies an 
LDP with good rate function
\[
I \;=\; \left\{ \begin{array}{ll} \displaystyle \frac{1}{2} \int_0^1 f'(s)^2 ds & \mbox{if }  f \in \mathcal{H}^1 \\
 & \\  +\infty & \mbox {otherwise} \end{array} \right.
\]
where $\mathcal{H}^1$ is the \emph{Cameron-Martin space}
$
 \{f \in \mathcal{AC}[0,T]: \; f(0)=0,\; \int_0^1 f'^2(s)ds < \infty\}
$
of absolutely continuous functions with square integrable derivatives.
$\mathcal{AC}[0,T]$ denotes the set of all real, absolutely continuous functions on $[0,T]$.
\end{theorem}


\subsection{The Ornstein-Uhlenbeck SDE and the time to exit from a deterministic boundary}
Consider the Ornstein-Uhlenbeck Stochastic Differential Equation (SDE)
\begin{equation} \label{OU-SDE-1}
dX_t = \mu X_t dt + \sigma dW_t, \hspace{0.2in} X_0 = x_0
\end{equation}
where $\mu >0$. Note that the expectation $\mathbb{E} X_t = x_0 e^{\mu t}$ increases 
exponentially with time and consider also the deterministic exponential function 
\begin{equation} \label{OU-SDE-2}
V(t) = v_0 e^{\beta t}  \;\;\mbox{ where } \;\; 0 \leq \beta < \mu \;\; \mbox{ and  } \;\; 0 < v_0 < x_0.
\end{equation}
Let
\begin{equation} \label{OU-SDE-4}
p_l(x_0,T)\;=\; \mathbb{P}(X_t > V(t); \; 0 \leq t \leq T)
\end{equation}
denote the probability that the process $\{X_t\}$ stays above the exponential boundary $V(t)$.
Here $1-p_l(x_0,T)$ may be thought of as a type of \emph{ruin probability}. We are interested
in evaluating $p_l(x_0,T)$ and the limiting probability 
$p_l(x_0):= \mathbb{P}(X_t > V(t); \; 0 \leq t )$. 
Due to the Markovian property of $\{X_t\}$, the ``non-ruin probability'' defined in (\ref{OU-SDE-4}) 
satisfies the PDE
\begin{eqnarray} \label{PDE-1}
&& \frac{1}{2} \sigma^2 f_{xx} + \mu x f_x + f_t \;=\;0, \;\; \mbox{ in }
D:=\{ (x,t): 0<t<T,\, x> v_0 e^{\beta t} \}\\ \nonumber
&& \; \mbox{ with boundary conditions $f(v_0e^{\beta t},t) = 0$
for $t \in [0,T]$ and $f(x,T) = 1$ for $x > v_0 e^{\beta T}$.}
\end{eqnarray}
We will not attempt to obtain an expression for the solution of (\ref{PDE-1}) due to
the difficulties that arise as a result of the shape of the domain $D$ though one may 
of course obtain numerical results for the ruin probability based on the above formulation. 
We will instead use Wentzell-Freidlin  ``low noise asymptotics'' in order to obtain a large
deviations estimate for the probability that $X_t$ crosses the path of $V(t)$ for
some $t \in [0,T]$ (Freidlin and Wentzell, 2012).

Similarly, if $\alpha >\mu$, $u_0 >x_0>0$, and $U(t) = u_0 e^{\alpha t}$ we will consider the
probability $p_u(x_0,T) = \bP(X_t < U(t);\, 0 \leq t \leq T)$ of remaining below the upper boundary
$U$, as well as the corresponding infinite horizon probability $p_u(x_0):= \bP(X_t < U(t);\, 0 \leq t )$.


\subsection{The Wentzell-Freidlin framework - Finite horizon problem} \label{sc:WF-FHP}

Wentzell-Freidlin theory generalizes the ideas in Schilder's theorem to the
paths of Stochastic Differential Equations.
To express the problem discussed in the previous section in the Wentzell-Freidlin
framework we consider the family  of processes $\{X_t^\epsilon\}$
\begin{eqnarray}   \label{SDEep1}
dX_t^\epsilon &=& \mu X_t^\epsilon dt + \sqrt{\epsilon} \, \sigma \, dW_t,
\hspace{0.15in} X_0^\epsilon = x_0
\end{eqnarray}
together with the deterministic process
\begin{eqnarray*}
\dot{x}(t) &=& \mu x(t)  , \hspace{0.15in} x(0)= x_0.
\end{eqnarray*}
Denote by $C[0,T]$ the set of continuous functions on $[0,T]$, and by $C_{x_0}[0,T]$
the set of all continuous functions $f:[0,T]\rightarrow \mathbb{R}$ with $f(0)=x_0$.
Consider the transformation $F: C[0,T] \rightarrow C_{x_0}[0,T]$ defined by
\begin{equation} \label{f-g}
f \;=\; F(g) \hspace{0.1in} \mbox{ with } \hspace{0.1in} f(t)
\;:=\; x_0 + \int_0^t \mu f(s) ds \;+\; \sigma g(t), \hspace{0.1in} t \in[0,T].
\end{equation}
Let $f_i$, denote the solution of (\ref{f-g}) when the driving function is $g_i\in C[0,T]$, $i=1,2$.
We may then establish the continuity of the map $F$ by means of a Gronwall argument
which shows that
\[
\Vert f_1 -f_2 \Vert \;\leq\;  \sigma \, e^{\mu T} \Vert g_1 -g_2 \Vert
\]
where $\Vert f \Vert:=\sup \{ \vert f(t) \vert: t\in [0,T]\}$ denotes the sup norm.
Theorem 5.6.7 of  Dembo and Zeitouni (2010) 
applies and therefore the solution of (\ref{SDEep1})
satisfies a Large Deviation Principle with good rate function
\begin{equation} \label{Action}
I(f,T) \; := \;  \left\{ \begin{array}{lll}\displaystyle \frac{1}{2}
\,\int_0^T \left( f'(t)-\mu f(t)\right)^2 \sigma^{-2} dt & \; &\mbox{ if }  f \in {\cal H}^1_{x_0}[0,T] \\
 & & \\ +\infty & & \mbox{ otherwise }  \end{array} \right.
\end{equation}
where ${\cal H}_{x_0}^1[0,T] := \{f:[0,T]\rightarrow \mathbb{R}\, , \; f(t) = x_0 + \int_0^t \phi(s)ds \;,
\; t\in[0,T], \phi\in L^2[0,T] \}$ is the Cameron-Martin space
of absolutely continuous functions with square integrable derivatives and initial value $f(0)=x_0$.

We can now state our first result.
\begin{theorem} \label{th:OU-1} In the above framework, if the lower boundary curve
is $V(t)=v_0e^{\beta t}$,
\begin{equation} \label{th1-1}
\lim_{\epsilon \rightarrow 0} \epsilon \, \log \mathbb{P}
\left(\min_{t\in[0,T]}X_t^\epsilon - V(t) \leq 0 \right) \;=\; - I_V(T). 
\end{equation}
The rate function $ I_V(T)$ is given by
\begin{equation} \label{OU-I}
I_V(T) \;=\; \frac{ \mu }{\sigma^2} \, \frac{\left( v_0 e^{\beta (T\wedge t_V^o)}
 - x_0 e^{\mu (T\wedge t_V^o)} \right)^2}{e^{2\mu (T\wedge t_V^o)} - 1}
\end{equation}
where $t^o_V$ is the unique positive solution of the equation
\begin{equation} \label{topt-v}
\phi_V(t) \; :=\; \left(1-\frac{\beta}{\mu}\right) e^{(\mu+\beta)t} + \frac{\beta}{\mu} e^{(\beta-\mu)t}
\;=\; \frac{x_0}{v_0}.
\end{equation}
Similarly, for the upper boundary curve $U(t)=u_0e^{\alpha t}$,
\begin{equation} \label{th1-2}
\lim_{\epsilon \rightarrow 0} \epsilon \, \log \mathbb{P}\left(\max_{t\in[0,T]}X_t^\epsilon - U(t) \geq 0 \right)
\;=\; - I_U(T) 
\end{equation}
with
\begin{equation} \label{OU-IU}
I_U(T) \;=\;   \frac{ \mu }{\sigma^2} \,  \frac{\left( u_0 e^{\alpha (T\wedge t_U^o)}
 - x_0 e^{\mu (T\wedge t_U^o)} \right)^2}{e^{2\mu (T\wedge t_U^o)} - 1}
\end{equation}
where $t^o_U$ is the unique positive solution of the equation
\begin{equation} \label{topt-u}
\phi_U(t) \; :=\; \frac{\alpha}{\mu} e^{(\alpha-\mu)t} - \left(\frac{\alpha}{\mu} - 1\right) e^{(\mu+\alpha)t}
 \;=\; \frac{x_0}{u_0}.
\end{equation}
\end{theorem}

\begin{proof}
The proof is long and will be divided into three parts for clarity of exposition.

\paragraph{Part 1.}
We begin by fixing $t>0$ and considering paths that start at $x_0$ at time $0$ and end at
$V(t)=v_0e^{\beta t}$ at time $t$. Consider the set
\[
\mathcal{H}_{x_0,V(t)}^1 \; := \;
\left\{h:[0,t]\rightarrow \mathbb{R} : h(s) = x_0 + \int_0^s \phi(u)du \;,
\; s\in[0,t],\; h(t)=V(t), \phi\in L^2[0,t] \right\}.
\]
Then, for $\eta >0$,
\begin{equation} \label{J_understar_1}
\lim_{\epsilon \rightarrow 0} \epsilon
\log \mathbb{P}\left( \sup_{0\leq s \leq t} \left| X_s^\epsilon - h(s) \right| < \eta \right) \;=\; - J_*(t).
\end{equation}
where $J_*(t)$  is the solution of the variational problem
\begin{equation} \label{Jstar}
J_*(t) := \inf\left\{ J(x;t): x \in \mathcal{H}_{x_0,V(t)}^1 \right\}
\end{equation}
with
\begin{equation} \label{rate-g}
J(x;t) = \int_0^t F(x,x',s) ds, \hspace{0.2in} 
\mbox{ and } \;\; F(x,x',s) = \frac{1}{2\sigma^2} \left( x' - \mu x \right)^2.
\end{equation}
$J(x;t)$ gives the rate function for a path $x(\cdot)$ that starts at $x_0$ and meets the lower boundary
at the point $(t,v_0 e^{\beta t})$ i.e.\ satisfies the boundary conditions
\begin{equation} \label{boundary-g}
x(0) = x_0, \hspace{0.3in} x(t) = v_0 e^{\beta t}.
\end{equation}
The infimum in (\ref{Jstar}) is taken over all absolutely continuous functions on $[0,t]$
with derivative in $L^2$. The function $x\in {\cal H}_{x_0,V(t)}^1[0,t]$ that minimizes
the integral defining the rate function is the solution of the Euler-Lagrange equation
(e.g.\ see Brechtken-Manderscheid, 1994)
\begin{equation} \label{Euler}
F_x - \frac{d}{ds} F_{x'} =0
\end{equation}
and the boundary conditions (\ref{boundary-g}).
With the given form of $F$ in (\ref{rate-g}) the Euler-Lagrange equation becomes
\begin{equation} \label{ODE}
x''(s) = \mu^2 x(s)
\end{equation}
which has the general solution
\begin{equation} \label{ch2:gensol1}
x(s) = c_1 e^{\mu s} + c_2 e^{-\mu s}.
\end{equation}
The values of $c_1$, $c_2$ for which $x$ satisfies the boundary conditions are given by
\begin{equation} \label{ch2:gensol2}
c_1 = \frac{ v_0 e^{\beta t} - x_0 e^{-\mu t}}{e^{\mu t} - e^{-\mu t}}, \hspace{0.3in}
c_2 = \frac{ x_0 e^{\mu t} - v_0 e^{\beta t}}{e^{\mu t} - e^{-\mu t}}.
\end{equation}
Thus (\ref{ch2:gensol1}) with the constants $c_1,c_2$ given by (\ref{ch2:gensol2}) gives the optimal path
\begin{eqnarray}  \label{optimal-path-finite}
x(s) &=&  \frac{v_0 e^{\beta t} \left(e^{\mu s}- e^{-\mu s} \right)
  + x_0 \left(e^{\mu (t-s)}- e^{-\mu (t-s)}\right)}{e^{\mu t} - e^{-\mu t}}  
 \;=\; \frac{v_0 e^{\beta t} \sinh(\mu s) + x_0 \sinh(\mu(t-s))}{\sinh(\mu t)}  . \;\;\;\;\;\;
\end{eqnarray}
From (\ref{ch2:gensol1}) $x'(s)-\mu x(s) = -2\mu c_2 e^{-\mu s}$ and, taking into account (\ref{rate-g}),
\[
J_*(t) \;=\; \frac{4 \mu^2 c_2^2}{2 \sigma^2} \,  \int_0^t e^{-2\mu s}ds \;=\;
\frac{\mu c_2^2}{\sigma^2} \left(1-e^{-2\mu t} \right).
\]
Using the expression for $c_2$ we have
\begin{equation} \label{rate2}
J_*(t) \;=\; \frac{ \mu}{\sigma^2}\,  \frac{\left( v_0 e^{\beta t} - x_0 e^{\mu t} \right)^2}{e^{2\mu t} - 1}.
\end{equation}

There remains to show that there is no path $x(s)$ with piecewise continuous derivative which
achieves a smaller value of the criterion, i.e.\ that the optimal solution does not have
corners. To this end we consider the Erdmann corner conditions (see Brechtken-Manderscheid,1994,  p.\ 33).
The first condition requires that $F_{x'}$ evaluated at the critical path be a continuous
function of $s$. Since $F_{x'} = \frac{1}{\sigma^2}\left(x'-\mu x\right)$ and $x(s)$ is
necessarily continuous, the first Erdmann condition implies the continuity of $x'(s)$ as well.
Therefore, by virtue of the first Erdmann condition alone we may conclude that the optimal
solution cannot have discontinuities in its derivative. For the sake of completeness we
mention that the second Erdmann condition requires that $F - x'F_{x'}$ evaluated at the
critical path $x(s)$ be also a continuous function of $s$. Since
$F - x'F_{x'} = -\frac{1}{2\sigma^2}\left((x')^2-\mu^2 x^2\right)$ and because of the
continuity of $x(s)$, this second condition by itself would allow the existence of corners
at which the first derivative changes sign. Such corners are of course precluded by the
first condition.

Finally, we claim that the solution we have found corresponds to a global minimum.
To see this we appeal to Theorem 3.16 of Brechtken-Manderscheid (1994),  p.\ 45,  
according to which it suffices to show that $F(x,x'):= \frac{1}{2\sigma^2}(x'-\mu x)^2$
(abusing slightly the notation)
is a convex function on $\mathbb{R}^2$. Indeed, we can show that, for any $(x_0,x_0') \in \mathbb{R}^2$,
\[
F(x,x') \geq F(x_0,x'_0) + F_x(x_0,x_0') \, (x-x_0) + F_{x'}(x_0,x_0') \, (x'-x'_0)
\]
or
\[
\frac{1}{2}\left(x'-\mu x\right)^2 \geq \frac{1}{2}\left(x_0'-\mu x_0\right)^2
- \mu\left(x_0'-\mu x_0\right)\, (x-x_0) + \left(x_0'-\mu x_0\right)\, (x'-x'_0) .
\]
This last inequality can be seen to be equivalent to
\[
\left(x'-\mu x\right)^2 + \left(x_0'-\mu x_0 \right)^2 \;
-\; 2\left(x'-\mu x\right)\left(x_0'-\mu x_0 \right) \geq 0
\]
which is clearly true and thus the convexity of $F$ and therefore the global
optimality of $x$ is established. Note however that this optimal solution was obtained
without taking into account the {\em path inequality constraint}
\begin{equation} \label{ipc}
x(s) \geq V(s) \;\;\; \mbox{ for all } s \in [0,t].
\end{equation}
Therefore the optimal path (\ref{optimal-path-finite}) may violate this inequality.

\paragraph{Part 2.}
In the first part we obtained the {\em fixed time optimal solution} under the boundary conditions
(\ref{boundary-g}).
In this part however we will solve the optimization problem with finite time horizon $[0,T]$,
\begin{equation} \label{part2-IT}
I(T) := \inf\{J(x,t): 0\leq t \leq T, x\in \mathcal{H}^1_{x_0,V(t)},
\mbox{ i.e. $x$ satisfies the conditions (\ref{boundary-g}) } \}
\end{equation}
still ignoring the {\em path inequality constraint} (\ref{ipc}).

Clearly $I(T) \;=\; \inf_{t \in [0,T]} J_*(t)$.
From (\ref{rate2}) we see that $J_*(t)$ is a continuously differentiable function for $t>0$.
We will establish that it is strictly convex on $[0,T]$. Indeed
\begin{equation} \label{J*derivative}
J_*'(t) \;=\;  \frac{2v_0\mu^2e^{\mu t}\left(x_0e^{\mu t}-v_0e^{\beta t}\right)}{\sigma^2 (e^{2\mu t} -1)^2} \,
\left[  \left( 1 - \frac{\beta}{\mu} \right) e^{(\beta+\mu)t}
+  \frac{\beta}{\mu}  e^{(\beta-\mu)t} - \frac{x_0}{v_0} \right].
\end{equation}
Given the definition of $\phi_V$ in (\ref{topt-v}) we note that the quantity inside the brackets above is
$\phi_V(t) -  \frac{x_0}{v_0}$
Since $0<\beta<\mu$ and $0<v_0<x_0$, $x_0e^{\mu t}-v_0e^{\beta t} >0$ for all $t\geq 0$
and thus the sign of $J_*'(t)$ is that of $\phi_V(t)- \frac{x_0}{v_0}$.
Note that $\phi'_V(t) = \frac{\mu-\beta}{\mu} e^{(\beta + \mu)t} \left[ \mu + \beta(1-e^{-2\mu t})\right]>0$
for all $t\geq 0$ and thus $\phi_V$ is strictly increasing. Also, given the definition of $\phi_V$ 
we have
$\lim_{t\rightarrow \infty}\phi_V(t) = +\infty$ and $\phi_V(0)=1$. Further, $\frac{x_0}{v_0} >1$
hence there exists a unique $t^o_V>0$ such that
\begin{equation} \label{eq-phi1}
\phi_V(t^o_V) \;=\; \frac{x_0}{v_0} >1 .
\end{equation}
In view of the expression (\ref{J*derivative}),
$J_*'(t) < 0$ for $0\leq t< t^o_V$, $J_*(t^o_V) =0$, and $J_*'(t)>0$ for $t>t^o_V$. Thus
$t^o_V$, the unique solution of (\ref{topt-v}), is a point of global minimum for $J_*$.
\begin{figure}[h!]
\begin{center}
\includegraphics[width=3.7in]{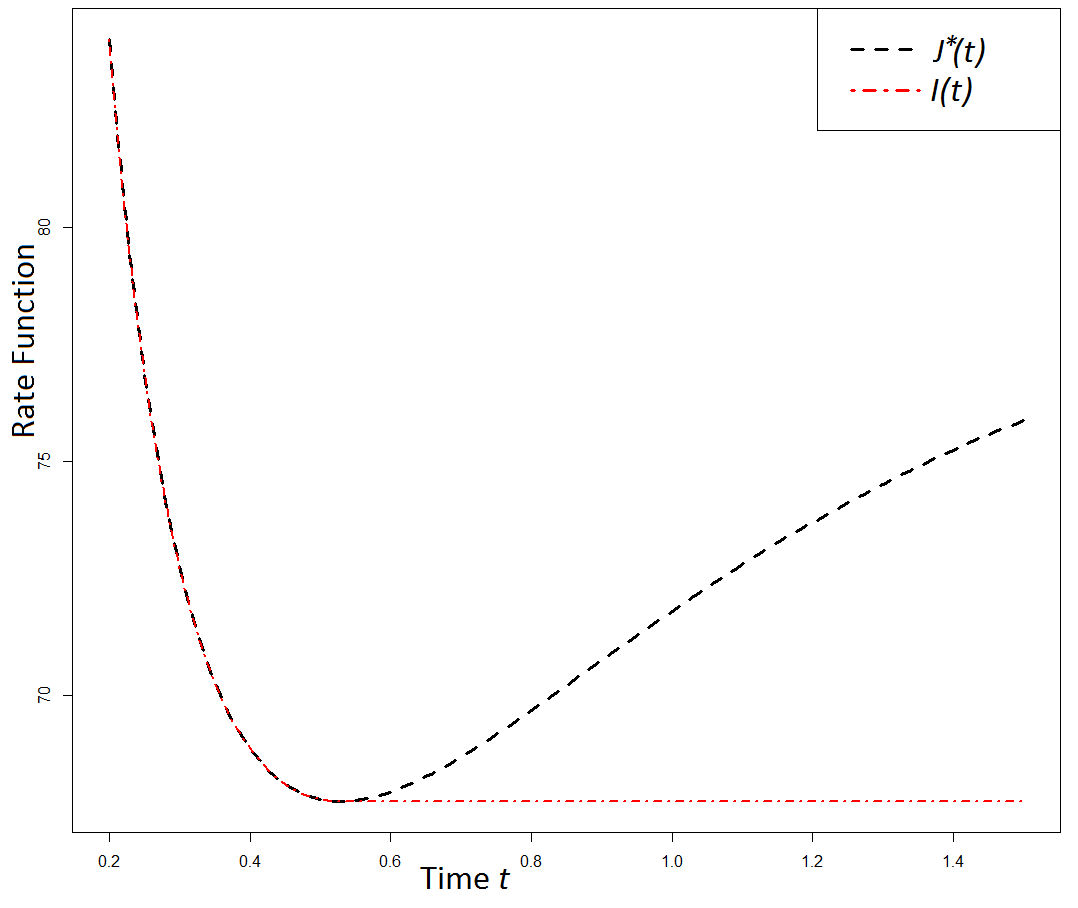}
\caption{The dotted black line denotes the function $J_*(t)$. The dotted red line denotes the rate function $I(t)$.
Here $\mu=2.5$, $\beta=1.0$, $x_0=4$, $v_0=1$ and $t^o_V \approxeq 0.529$.}
\label{fig:rate-1}
\end{center}
\end{figure}
Then
\begin{equation} \label{IT-part2}
I(T)\;=\; \inf_{t \in [0,T]} J_*(t) \;=\;
\left\{ \begin{array}{ccc} J_*(T) & \mbox{ if } & T\leq t^o_V \\ J_*(t^o) & \mbox{ if } & T > t^o_V \end{array} \right. \;.
\end{equation}
Figure \ref{fig:rate-1} illustrates the behavior of the function $J_*(t)$ and that of $I(t)$.


\paragraph{Part 3.}
We complete the proof by showing that the optimal rate given by (\ref{IT-part2})
remains valid even after taking into account the additional
path inequality constraint (\ref{ipc}).
Define
\begin{equation} \label{Jstarstar}
J_{**}(t) := \inf\left\{ J(x;t): x \in \mathcal{H}_{x_0,V(t)}^1, x(s) \geq V(s) \mbox{ for } s \in [0,t]. \right\}
\end{equation}
Consider the optimal path $x(s)$ of Part 1 given in (\ref{ch2:gensol2}) , (\ref{ch2:gensol1}),
(\ref{optimal-path-finite}), for all $s \geq 0$.
Note that $c_2>0$ (since $\mu >\beta$ and $x_0 > v_0$). The sign of $c_1$ depends on $t$:
$c_1>0 \Leftrightarrow v_0 e^{(\mu+\beta)t} -x_0 >0 $ and this is equivalent to
\begin{equation} \label{inequality-condition-1}
 t > t_1:= \frac{1}{\mu+\beta} \, \log\frac{x_0}{v_0}.
\end{equation}
We also point out that
\begin{equation} \label{t1to}
t_1 < t^o_V.
\end{equation}
This follows by the fact that $\phi_V$ is a strictly increasing function and
\[
\phi_V(t_1) \;=\; \left( 1 - \frac{\beta}{\mu} \right) e^{(\beta+\mu)t_1}
+  e^{-2\mu t}\frac{\beta}{\mu}  e^{(\beta+\mu)t_1} \;=\;
\frac{x_0}{v_0}  \left(  1 - \frac{\beta}{\mu}( 1-   e^{-2\mu t} ) \right) \;<\;  \frac{x_0}{v_0} = \phi_V(t^o_V).
\]

We distinguish three cases according to the relationship between $t$ and $t_1$.

\noindent
\underline{\em Case 1: $t <t_1$.} This implies that $c_1<0$. Because $x(0)>V(0)=v_0$,
$x'(s) = \mu c_1 e^{\mu s} - \mu c_2 e^{-\mu s}<0$ for all $s\geq 0$, and
$\lim_{s \rightarrow \infty} x(s) = -\infty$, $t$ is the unique intersection point of the
paths $x(\cdot)$ and $V(\cdot)$ and the inequality constraint (\ref{ipc}) is satisfied.


\noindent
\underline{\em Case 2: $t = t_1$.} Then, from (\ref{ch2:gensol2}) $c_1=0$ and $c_2=x_0$ and hence
$x(s) = x_0 e^{-\mu s}$. Again, the paths $x(\cdot)$ and $V(\cdot)$ intersect only once, at $t$, $x(s) > V(s)$
for $s\in[0,t)$, and the path inequality constraint is satisfied.

\noindent
\underline{\em Case 3: $t > t_1$.} Here both $c_1>0$ and $c_2>0$ and thus $x(s)>0$ 
for all $s>0$. Therefore, as a result of (\ref{ODE}), $x''(s)>0$ and the function $x$ is 
strictly convex for all $s \geq 0$. In this case, as is shown in the Appendix, 
the paths $x(\cdot)$ and $V(\cdot)$ 
intersect at precisely two points, one of which is of course
$t$ while the other will be denoted by $\tau(t)$.
\begin{figure}[h!]
\begin{center}
\includegraphics[width=6.8in]{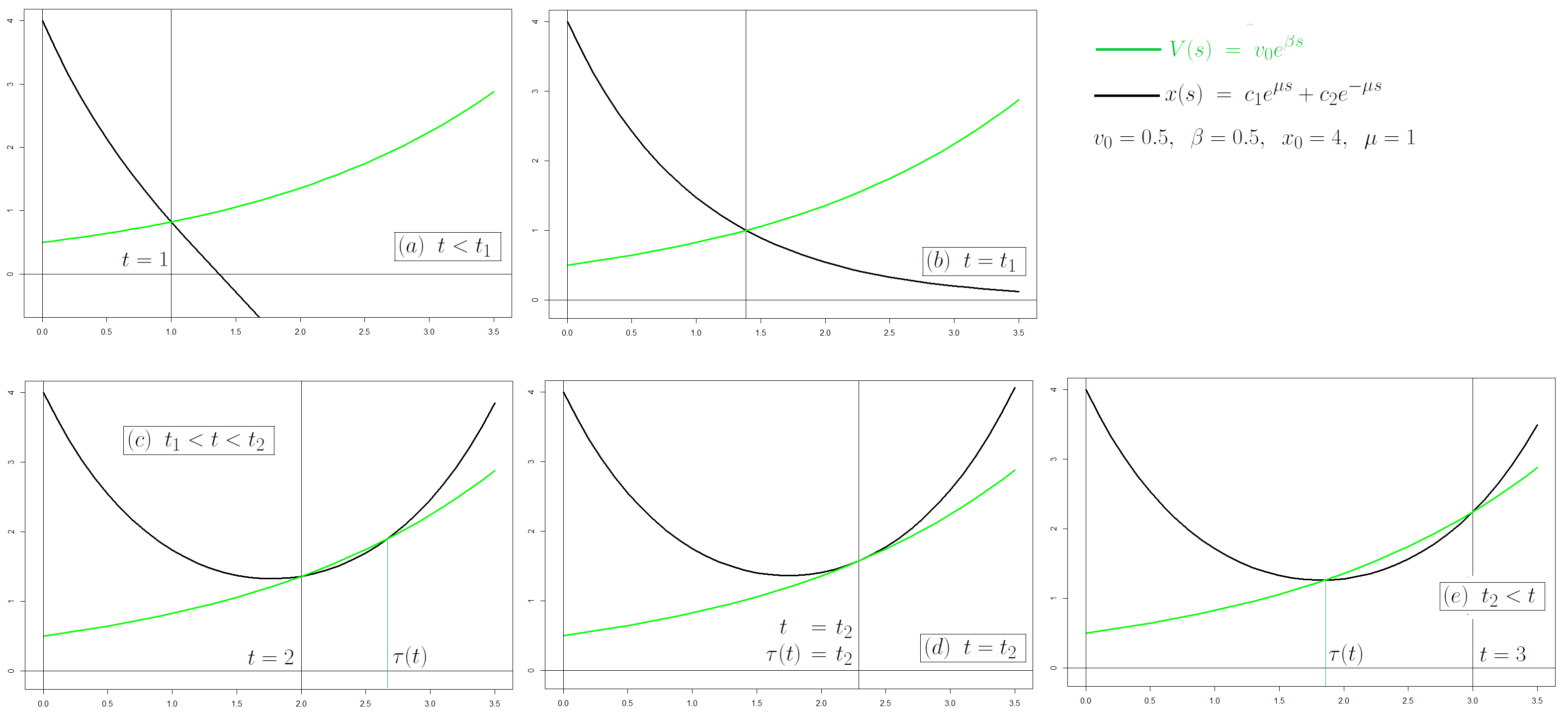}
\caption{The three cases. In this example $\mu=1$, $\beta=0.5$, $x_0=4$, $v_0=0.5$. These give
$t_1 \approx 1.39$ and $t_2 \approx 2.29$. (a) shows the behavior when $t=1<t_1$. This is case 1 and
$x(\cdot)$ decreases monotonically. The inequality constraints are satisfied. In (b) $t=t_1$. 
This is case 2 and again the inequality conditions are satisfied. The remaining tree plots 
illustrate case 3. In (c)  $t =2 < t_2$ and in (d)  $t=t_2\approx 2.29$. Since in these two cases
$t \leq \tau(t)$ the inequality constraints are satisfied. In (e) however, where $t>\tau(t)=t_2$
the inequality constraint (\ref{ipc}) is not satisfied.}
\label{fig:hitting-examples}
\end{center}
\end{figure}
Figure \ref{fig:hitting-examples} illustrates this for specific values of the parameters $\mu$, $\beta$, $x_0$, $v_0$. 
For the values of the parameters in Figure \ref{fig:hitting-examples} $t_1 = \frac{2}{3}\log 8 \approx 1.39$. In this
figure, in the upper left, the path $x(s)$ is decreasing and eventually becomes negative. There is a single 
intersection between the curves $x(s)$ and $V(s)$. On the other hand in the graph lower, in the middle,
($t=2$) and in the right ($t=3$) the path $x(s)$ is strictly convex, as is $V(s)$, and thus the two curves 
intersect in two points. For $t=2$ the path $x(s)$ satisfies (\ref{Derivative-Condition}) and therefore
(\ref{ipc-d}) and the path inequality constraint (\ref{ipc}) while for $t=3$ it does not.

The key remark is the following: If $x'(t) < V'(t)$ then the path $x(\cdot)$ intersects $V(\cdot)$ from above at $t$,
then again from below at $\tau(t)>t$. If, conversely, $x'(t)> V'(t)$ then $x(\cdot)$ intersects $V(\cdot)$ from
below at $t$. Since $x(0) > V(0)$ this necessarily implies that there was an earlier crossing from above at $\tau(t) < t$.
(The case $x'(t) = V'(t)$ corresponds to $t=\tau(t)$. The path $x(\cdot)$ is tangent to $V(\cdot)$ at $t$ and
$x(s) > V(s)$ for all $s \ne t$.)

The situation in Case 3 is examined in more detailed in Section \ref{sec:app-t2} of the Appendix where it is established
that there exists a time $t_2$ such that $t_1<t_2$ and the relationship between $t$ and $t_2$ determines whether
the path $x(\cdot)$ satisfies the inequality constraints (\ref{ipc}) or not. Specifically
\begin{itemize}
\item[] If $t_1<t<t_2$ then $x(\cdot)$ intersects $V(\cdot)$ from above at $t$ and
hence it satisfies the inequality constraint $x(s) > V(s)$ for $s \in [0,t)$. It crosses $V(t)$ once again at $\tau(t) >t$,
this time from below.
\item[] If $t=t_2$ then $x(t)$ is tangent to $V(t)$ at $t$. It satisfies the 
inequality constraint $x(s) > V(s)$ for $s\in [0,t)$ (and
in fact even beyond $t$ though this is of no interest for our purposes).
\item[]
If $t>t_2$ then $x(t)$ crosses $V(t)$ from below. This means that there was a first crossing from above at $\tau(t)<t$.
As a result $x(s) < V(s)$ when $s \in (\tau(t),t]$ and the inequality constraint (\ref{ipc}) is not satisfied in this case.
\end{itemize}

\begin{figure}[h!]
\begin{center}
\includegraphics[width=6in]{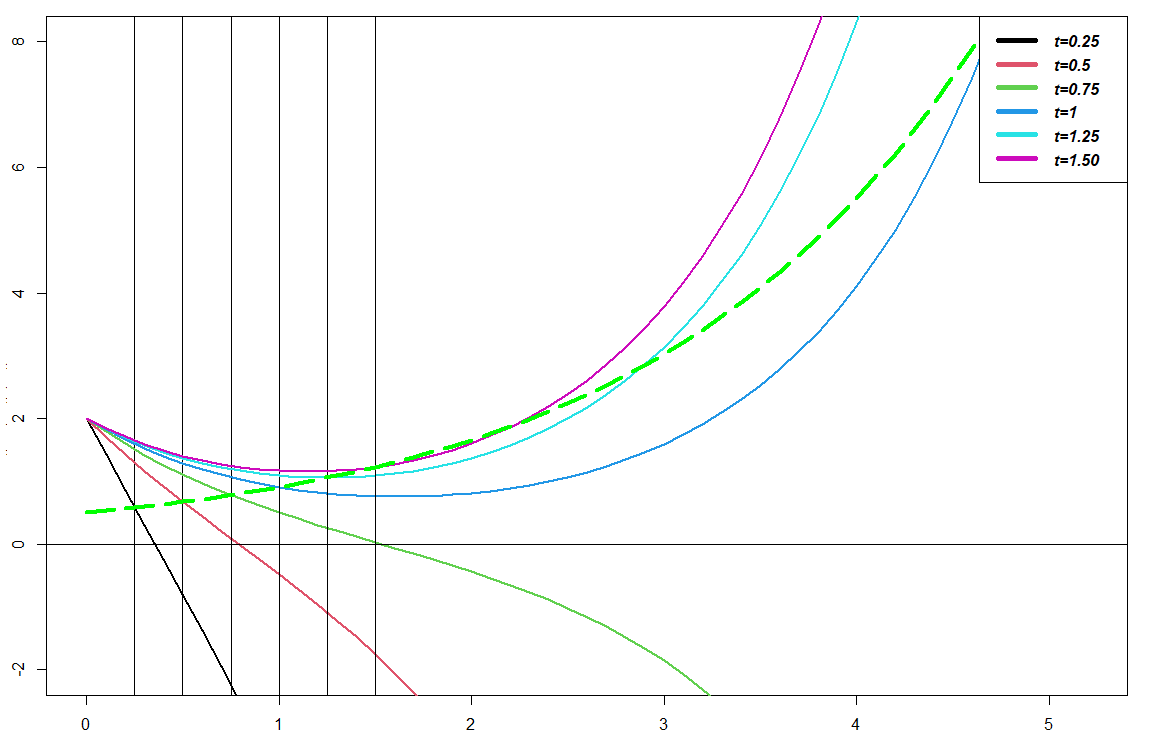}
\caption{Here $\mu=1$, $\beta=0.6$, $x_0=2$, $v_0=0.5$. Thus $t_1=0.8664$. 
The hitting times range from $t=0.25$ to $t=1.50$.
Note that, for $t=0.25, \, 0.50,$ and $0.75$ the path $x(s)$ eventually becomes negative, after hitting $V(s)$,
the thick green line.  When the hitting times are greater than $t_1$,
(i.e.\ $t=1,\,1.25,\,$ and $1.5$) the path $x(s)$ is a convex function
and has two intersection points with the dotted green line.}
\label{fig:multi-hitting-times-2}
\end{center}
\end{figure}
Figure \ref{fig:multi-hitting-times-2} illustrates these cases. For $t=0.25$, $0.5$,
and $0.75$ (black, red, and green paths) the paths eventually become negative
and intersect the dotted green line (i.e.\ the function $V(\cdot)$) once.
In the rest of the cases the paths remain positive and intersect the dotted green line twice.

Thus the optimal path of Part 1 also satisfies the constraint (\ref{ipc}) iff $t\leq t_2$. In that case
the path given by (\ref{optimal-path-finite}) minimizes the functional $J(x,t)$ in (\ref{rate-g}) under
the boundary conditions (\ref{boundary-g}) and the path inequality constraints (\ref{ipc}). Then
\begin{equation} \label{t2holds}
J_{**}(t) = J_*(t) \hspace{0.1in} \mbox{ when } \; t<t_2.
\end{equation}


If $t>t_2$ then $t$ is {\em the second point of intersection of $x(s)$ with $V(s)$} and (\ref{ipc}) 
is not satisfied. This means that the path $x(s)$ is not feasible under the additional constraint 
$x(s)>V(s)$ and therefore that the minimum value $J_*(t)$ obtained without taking into account 
the inequality constraint is smaller than $J_{**}(t)$. Thus we have
\begin{equation} \label{rr3}
\begin{array}{ccc} J_{**}(t)  &=&  J_*(t) \;\mbox{ if } t\leq t_2 \\ J_{**}(t) &\geq&  J_*(t) \;\mbox{ if } t > t_2 \end{array}
\end{equation}



Then, the rate function in (\ref{th1-1}), defined as
\begin{equation}  \label{IV}
I_{V}(T) \; := \; \inf\left\{ J(x;t): x \in \mathcal{H}_{x_0,V(t)}^1, x(s) \geq V(s) \mbox{ for } 0 < s <t, \;\; 0 < t \leq T. \right\}
\end{equation}
can be obtained as
\begin{equation} \label{IVmin}
I_{V}(T) \; := \; \min_{t \in (0,T]} J_{**}(t).
\end{equation}
If $T\leq t_2$ then $ J_{**}(t) = J_*(t)$ and hence $I_V(T) = \min_{t\in (0,T]} J_*(t) = J_*(t^o_V\wedge T)$
due to the fact that $J_*$ is strictly decreasing in $(0,t^o_V)$ and strictly increasing in $(t^o_V,\infty)$.

Therefore we conclude that $I_V(T)$ is also given by (\ref{rate2}). This concludes the proof of the first part of
Theorem \ref{th:OU-1}. The proof of the second part, pertaining to the upper boundary curve, is similar and will be omitted.
\end{proof}


\subsection{The infinite horizon problem -- lower and upper boundary curves} \label{sec:ihp-lub}

We now turn to the infinite horizon problem of obtaining a large deviations estimate for
the probability $\bP(\inf_{t\geq 0} X_t - v_0 e^{\beta t} \leq 0 )$ and
$\bP(\inf_{t\geq 0} X_t - u_0 e^{\alpha t} \geq 0 )$ in the same context as
that of the previous section.
It is of course possible, after having solved the corresponding finite horizon problem 
as we saw in the previous section, to minimize over the horizon $T$. Instead of this,
we will use here the standard {\em transversality conditions} approach of the Calculus 
of Variations in order to tackle in one step the infinite horizon problem by considering 
variational problems with variable end-points.

\begin{theorem}
Suppose $\{X_t^\epsilon\}$, $\epsilon>0$, is the family of diffusions described by the solution
of the SDE (\ref{SDEep1}). Suppose also that the upper bounding curve $U(t)=u_0e^{\alpha t}$
and lower bounding curve $V(t) = v_0 e^{\beta t}$ satisfy the inequalities $v_0 < x_0 < u_0$ and
$\beta < \mu < \alpha$. Then
\begin{itemize}
\item[a)] The probability of ever hitting the lower boundary satisfies
\begin{equation} \label{inf_lb_rate}
- \lim_{\epsilon \rightarrow 0} \epsilon \log \bP\left(\inf_{t\geq 0} X_t^\epsilon - v_0 e^{\beta t} \leq 0 \right)
\;=: \; I_{V}(\infty) \;=\;  \frac{x_0^2 \mu}{\sigma^2} \, \,
\frac{1-e^{-2\mu t^o_V}}{\left(1+\frac{\beta}{\mu-\beta}\,e^{-2\mu t^o_V} \right)^2}
\end{equation}
and $t^o_V$ is the unique root of equation (\ref{topt-v}).
The optimal path $x_*$ hitting the lower bound is given by
\begin{equation} \label{OU-optimal-path}
x_*(t) \;=\; x_0 \frac{ e^{-\mu(t^o_V-t)} + \left(\frac{\mu}{\beta} - 1\right) e^{\mu(t^o_V-t)}}{ e^{-\mu t^o_V}
+ \left(\frac{\mu}{\beta} - 1\right) e^{\mu t^o_V}} .
\end{equation}
\item[b)] The probability of ever hitting the upper boundary satisfies
\begin{equation} \label{inf_ub_rate}
- \lim_{\epsilon \rightarrow 0} \epsilon \log \bP\left(\inf_{t\geq 0} X_t^\epsilon - u_0 e^{\alpha t} \geq 0 \right)
\;=\; I_{U}(\infty) \;=\; \frac{x_0^2 \mu}{\sigma^2} \, \,
\frac{1-e^{-2\mu t^o_U}}{\left(1+\frac{\alpha}{\mu-\alpha}\,e^{-2\mu t^o_U} \right)^2}
\end{equation}
and $t^o_U$ is the unique root of the equation (\ref{topt-u}). The optimal path hitting the upper bound is given by
\begin{equation} \label{u-opt-path-1}
x^*(t) \;=\; x_0 \,
\frac{ e^{-\mu(t^o_U-t)} - \left(1- \frac{\mu}{\alpha} \right) e^{\mu(t^o_U-t)}}
{ e^{-\mu t^o_U} - \left(1- \frac{\mu}{\alpha} \right) e^{\mu t^o_U}} .
\end{equation}
\end{itemize}
\end{theorem}

\begin{proof}
Consider first the problem of hitting the upper boundary at some time $T_U$ before hitting the lower boundary. 
We will obtain low noise logarithmic asymptotics for the probability of hitting the upper boundary (without having 
first hit the lower). Because in the limit, as $\epsilon \rightarrow 0$, the probability of ever hitting either the upper 
or the lower boundary goes to 0 exponentially (in $\frac{1}{\epsilon}$) we expect that the presence of the lower 
boundary (and the stipulation to avoid it) does not affect the probability of hitting the upper boundary.

The optimization problem for the action functional becomes
\begin{eqnarray} \label{action_functional}
&& \min \int_0^{T_U} F(x,x',t)dt, \hspace{0.1in}  \mbox{with }
F(x,x',t) = \frac{1}{2\sigma^2} \left( x' - \mu x \right)^2, \hspace{0.2in} \\ \nonumber
&& \mbox{subject to the constraints } \\ \label{end_constraints}
&&  x(0)=x_0, \;\; \mbox{ and } x(T_U)=U(T_U) \\  \label{path_ineq_upper}
&&  V(t)< x(t) < U(t) \; \mbox{ for } 0\leq t < T_U, \;\;
\end{eqnarray}

In the above, both the optimal path $x$ and the horizon $T_U$ are unknowns to be determined.
Our approach to dealing with the {\em inequality path constraint,} (\ref{path_ineq_upper})
$x(t)>V(t)$ for all $t \in [0,T)$ will be to initially ignore it and obtain an optimal hitting time 
$t^o_U$ and an optimal path $x_*$ minimizing the criterion (\ref{action_functional}) and satisfying the
boundary conditions (\ref{end_constraints}).  We will then show that this optimal path satisfies the constraints
(\ref{path_ineq_upper}).

The necessary conditions for a minimum in the problem {\em without the path inequality constraint} are
\begin{eqnarray}
&\mbox{Euler-Lagrange Equation:}& \;\; F_x - \frac{d}{dt} F_{x'} = 0, \label{EL} \\
&\mbox{Boundary Conditions:}& \;\; x(0)=x_0, \;\;\; x(T_U)=U(T_U), \ \label{B1-OU} \\
&\mbox{Transversality Condition:}& \;\; F + (U'-x')F_{x'} =0 \;\; \mbox{ at } T_U. \label{Transversality-OU}
\end{eqnarray}
Taking into account that $F_{x}  = - \mu \sigma^{-2} \left( x' -\mu x \right)$, $F_{x'} =  \sigma^{-2} \left( x' -\mu x \right)$,
$\frac{d}{dt} F_{x'} =  \sigma^{-2} \left( x'' -\mu x' \right)$, the Euler-Lagrange equation becomes
\[
 F_x - \frac{d}{dt} F_{x'} = - \sigma^{-2} \left( x'' -\mu^2 x \right) = 0
\]
and thus
\begin{equation} \label{EulerDE}
x'' - \mu^2 x \;=\;0.
\end{equation}
This has the general solution
\begin{equation} \label{xt-OU}
x(t) = C_1 e^{\mu t} + C_2 e^{-\mu t} .
\end{equation}
Taking into account the boundary conditions (\ref{B1-OU}), we obtain
\begin{eqnarray} \label{B2}
x(0)&=& C_1+C_2 \;=\; x_0, \\ \label{B3}
x(T)&=& C_1 e^{\mu T_U} + C_2 e^{-\mu T_U} \;=\; u_0 e^{\alpha T_U} .
\end{eqnarray}
The transversality condition (\ref{Transversality-OU}) gives
\[
\frac{1}{2\sigma^2} \left(x'(T_U)-\mu x(T_U) \right)^2 
+ \left(u_0 \alpha e^{\alpha T_U} - x'(T_U) \right) \frac{1}{\sigma^2}
\left(x'(T_U) -\mu x(T_U) \right) \;=\; 0
\]
or
\begin{equation} \label{product}
 \left(x'(T_U) -\mu x(T_U) \right) \, \left(  -x'(T_U)-\mu x(T_U) + 2u_0 \alpha e^{\alpha T_U}  \right)
  \;=\; 0.
\end{equation}
Taking into account (\ref{xt-OU}), it follows that $x'(T_U)-\mu x(T_U) = -2\mu C_2 e^{-\mu T_U}$ 
and hence, if the first factor of (\ref{product}) were to vanish, this would imply that $C_2 =0$. 
This in turn implies, in view of (\ref{xt-OU}), (\ref{B2}), and
(\ref{B3}), that $x(T) = x_0 e^{\mu T_U} = u_0 e^{\alpha T_U}$ which is impossible
since $x_0 < u_0$ and $\mu < \alpha$. Hence (\ref{product}) implies
\begin{equation} \label{product-1}
u_0 e^{\alpha T_U} \;=\; \frac{\mu}{\alpha}\, C_1 e^{\mu T_U}.
\end{equation}
From (\ref{B1-OU}) and (\ref{product-1}) we obtain
\begin{eqnarray*}
C_1 + C_2 &=& x_0 \\
C_1 \left( 1 -  \frac{\mu}{\alpha} \right) \, e^{\mu T_U} + C_2 e^{-\mu T_U} &=& 0
\end{eqnarray*}
whence it follows that
\begin{equation} \label{C12}
C_1 \;=\; \frac{x_0 e^{-\mu T_U}}{\left(\frac{\mu}{\alpha} - 1 \right) \, 
e^{\mu T_U} + e^{-\mu T_U}}, \hspace{0.3in}
C_2 \;=\; \frac{x_0 \left(\frac{\mu}{\alpha} - 1 \right)
\, e^{\mu T_U}}{\left(\frac{\mu}{\alpha} - 1 \right) \, e^{\mu T_U} + e^{-\mu T_U}}.
\end{equation}
From (\ref{xt-OU}) and (\ref{product-1}) we obtain the equation
\begin{equation} \label{u-transversality-OU-1}
\left(\frac{\alpha}{\mu} -1 \right) \, e^{(\mu + \alpha)T_U} \,-\,  \frac{\alpha}{\mu}
\,e^{(\alpha-\mu) T_U} + \frac{x_0}{u_0} \;=\; 0
\end{equation}
which must be satisfied by the optimal hitting time $T_U$. This is precisely equation (\ref{topt-u})
and therefore 
\[
T_U=t_U^o,
\]
the unique solution of (\ref{topt-u}). 
(Indeed, it is easy to see that (\ref{topt-u}) has a unique solution: We have $\phi_U(0)= 1$, 
$\lim_{t\rightarrow \infty} \phi_U(t) = - \infty$, and $\phi_U'(t)
= - \frac{\alpha -\mu}{\mu} e^{(\mu+\alpha)t}
\left[ \mu + \alpha(1- e^{-2\mu t}) \right] <0$ for all $t \geq 0$.)
{\em Note in particular that the value of $T_U$ does not depend on $\sigma$ as is clear from 
(\ref{u-transversality-OU-1}).}
Also, since $T_U$ is the unique solution of (\ref{topt-u}), it
must satisfy the equation $(\mu-\alpha)e^{2\mu T_U} + \alpha = e^{-\alpha T_U} \mu \frac{x_0}{u_0}$.
Since the right hand side of this equation is strictly positive, the left must be as well and this implies
\begin{equation} \label{TU-inequality}
e^{2\mu T_U} < \frac{\alpha}{\alpha-\mu} \;\; \mbox{ or } \;\; T_U < \frac{1}{2\mu} \log\frac{\alpha}{\alpha-\mu}.
\end{equation}

An alternative expression for $C_1$, $C_2$, taking into account (\ref{u-transversality-OU-1}) is
\begin{equation} \label{C12-alt}
C_1 \;=\; u_0 \frac{\alpha}{\mu} e^{(\alpha-\mu)T_U}  , \hspace{0.3in}
C_2 \;=\; u_0 \left( 1-\frac{\alpha}{\mu}\right)  e^{(\alpha+\mu)T_U}  .
\end{equation}
Thus, by (\ref{xt-OU}) and (\ref{C12}) or (\ref{C12-alt}), the following equivalent expressions 
describe the optimal path.
\begin{equation} \label{u-opt-path-2}
x^*(t) \;=\; x_0 \,
\frac{\left(\frac{\alpha}{\mu}-1\right)  e^{\mu (T_U-t)} - \frac{\alpha}{\mu} e^{-\mu (T_U-t)}}
{\left(\frac{\alpha}{\mu}-1\right)  e^{\mu T_U} - \frac{\alpha}{\mu} e^{-\mu T_U}} 
\;= \; 
u_0 e^{\alpha T_U} \left[ \frac{\alpha}{\mu}  e^{-\mu(T_U-t)}
- \left( \frac{\alpha}{\mu} - 1 \right) e^{\mu(T_U-t)}\right].
\end{equation}

From the above we obtain the rate function $I_U$ given in (\ref{inf_ub_rate}). Alternative 
expressions for the rate $I_U$, using (\ref{u-transversality-OU-1}) are, of course, possible. 
For instance,
\begin{equation} \label{alt-III}
I_U \;=\; \frac{\mu}{\sigma^2}
\frac{\left(u_0 e^{(\alpha-\mu)T_U} - x_0\right)^2}{1-e^{-2\mu T_U}}  \;\;=\;\;
\frac{\mu}{\sigma^2}u_0^2
\left( 1-\frac{\mu}{\alpha}\right)^2 e^{2\alpha T_U} \left(e^{2\mu T_U} - 1 \right).
\end{equation}

There remains to show that the optimal path obtained in (\ref{u-opt-path-2}) also satisfies the 
inequality constraints $v_0 e^{\beta t} < x^*(t) < u_0 e^{\alpha t}$ for $t \in [0,T_U)$.
Indeed
\begin{eqnarray*}
x^*(t) - x_0 e^{\mu t} &=&  
\frac{2\left( 1- \frac{\mu}{\alpha}\right)  \sinh \mu t}{ e^{-2\mu T_U} - 1+ \frac{\mu}{\alpha} } 
 \;>\; 0 \;\; \mbox{ for } t>0
\end{eqnarray*}
the denominator in the expression above being positive by (\ref{TU-inequality}).
Since $v_0 e^{\beta t} < x_0 e^{\mu t}$ for all $t>0$ the above inequality implies 
$x^*(t) > v_0e^{\beta t} = V(t)$ for $t \in [0,T_U)$.

Next, we will show that $x^*(t)< U(t)$ for all $t  \in [0, T_U)$. Since $\psi(t):=U(t)-x^*(t)$ satisfies
$\psi(0)=u_0 - x^*(0)=u_0-x_0 >0$, $\psi(T_U)=0$, and $\psi$ is continuous it suffices to prove that $\psi$
does not vanish in the interval $[0,T_U)$. 

By (\ref{u-opt-path-2}) $\psi(t)=0$ is equivalent to 
\begin{equation} \label{constraint-1}
x_0\big[ \big(\frac{\alpha}{\mu}-1\big)  e^{\mu (T_U-t)} - \frac{\alpha}{\mu} e^{-\mu (T_U-t)}\big] = 
u_0e^{\alpha t} \big[ \big(\frac{\alpha}{\mu}-1\big)  e^{\mu T_U} - \frac{\alpha}{\mu} e^{-\mu T_U}\big] .
\end{equation}
The right hand side of the above can be written as
\[
u_0e^{-\alpha (T_U- t)} \big[ \big(\frac{\alpha}{\mu}-1\big)  e^{(\alpha+\mu) T_U} 
- \frac{\alpha}{\mu} e^{(\alpha-\mu) T_U}\big] 
= - x_0 e^{-\alpha (T_U- t)} 
\]
the last equality following from (\ref{u-transversality-OU-1}). Hence  setting $T_U-t =:s$ and 
taking into account (\ref{constraint-1}) we obtain
\[ 
(\frac{\alpha}{\mu}-1\big)  e^{\mu s} - \frac{\alpha}{\mu} e^{-\mu s}  + e^{-\alpha s} =0.
\]
It is easy to see that this equation does not have strictly positive solutions in $s$ and thus there 
exists no $t \in [0,T_U)$ satisfying (\ref{constraint-1}).
Therefore the critical path $x^*(t)$ satisfies  the inequality $x^*(t) < U(t)$ as well, for all $t\in [0,T)$.

The proof for the corresponding results for the lower boundary are established in an entirely similar manner.
\end{proof}
\begin{figure}
\begin{center}
\includegraphics[width=3.0in]{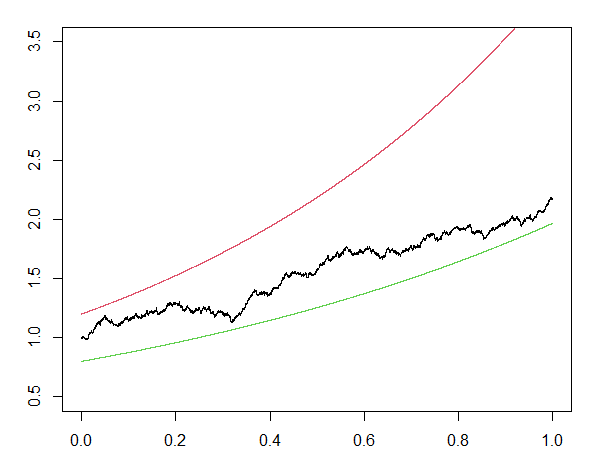}
\caption{An Ornstein-Uhlenbeck process evolving between an upper and a lower exponential bound.}
\label{fig:r-2}
\end{center}
\end{figure}

\begin{figure}[h!]
\begin{center}
\includegraphics[width=3in]{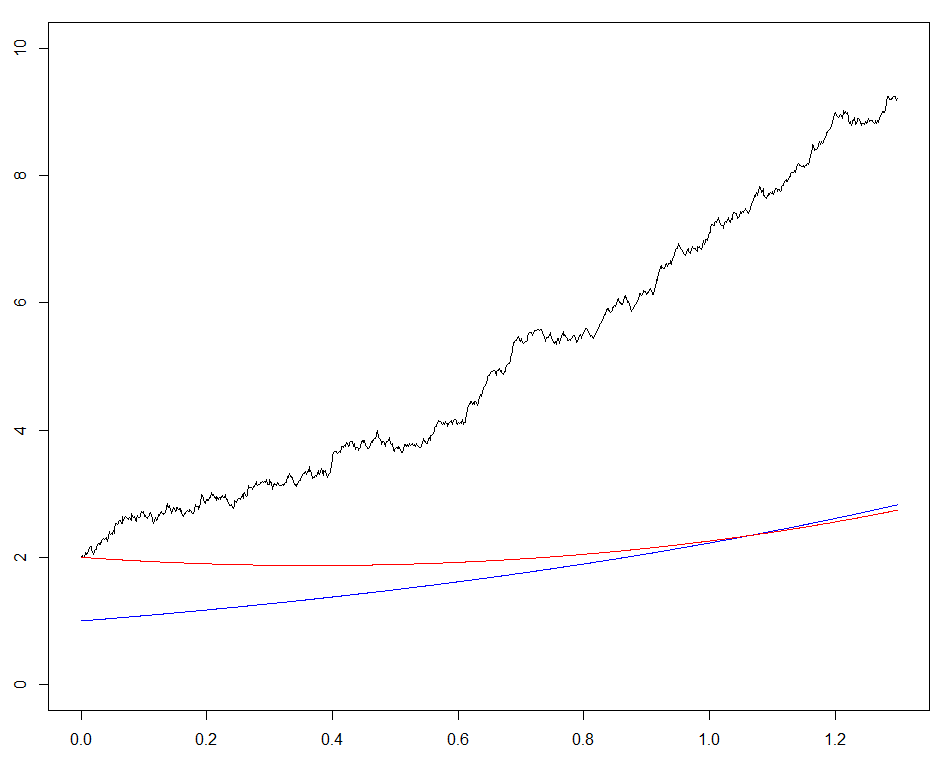}
\caption{The black line is a typical path of an OU process with $\mu=1$, $\sigma=1$
and starting point $x_0=2$. The blue curve is the lower
exponential bound $v_0e^{\beta t}$ with $v_0=1$ and $\beta =0.8$. 
The meeting time $T$ obtained by solving numerically (\ref{topt-v})
is equal to 1.0621. Finally the red optimal (large deviation) path is obtained from (\ref{OU-optimal-path})}
\label{fig:r-3}
\end{center}
\end{figure}
Intuitively, the uniqueness of the solution of (\ref{u-transversality-OU-1}) makes sense.
If $T_U$ is very small the noise factor $W_t$ must exhibit an
extremely unlikely behavior in order for the OU process to rise to the level of the upper curve $U(t)$.
So having more time available makes the rare event of hitting the upper boundary more likely. But if $T_U$ is too
large, because of the difference in the rates of the two processes, again hitting the upper boundary becomes
extremely unlikely. Also, in some cases, in the infinite horizon problem, an infimum may exist but no minimum.
The rate function $I$ is not "good" and compactness fails. In practical terms, the more time available the more
likely it is that the noise term will cause the diffusion path to hit the deterministic boundary curve.

In Figures \ref{fig:r-4}, \ref{fig:r-5}, we consider the OU process $dX_t =  X_t + dW_t$, with $X_0 = x_0$,
(with the value of the parameters $\mu= 1$, $\sigma=1$) and the lower  and upper bounds
$v(t)=0.5 e^{0.5 t}$, $u(t) = 2 e^{1.3t}$. (Thus $\alpha=1.3$, $u_0=2$, $\beta=0.5$ and $v_0=0.5$.)
In Figure \ref{fig:r-3} the optimal value of $T$ that corresponds to the solution of the optimization problems
of section \ref{sec:ihp-lub}  (equations (\ref{topt-v}) and (\ref{topt-u})).
\begin{figure}[h!]
\begin{center}
\includegraphics[width=3in]{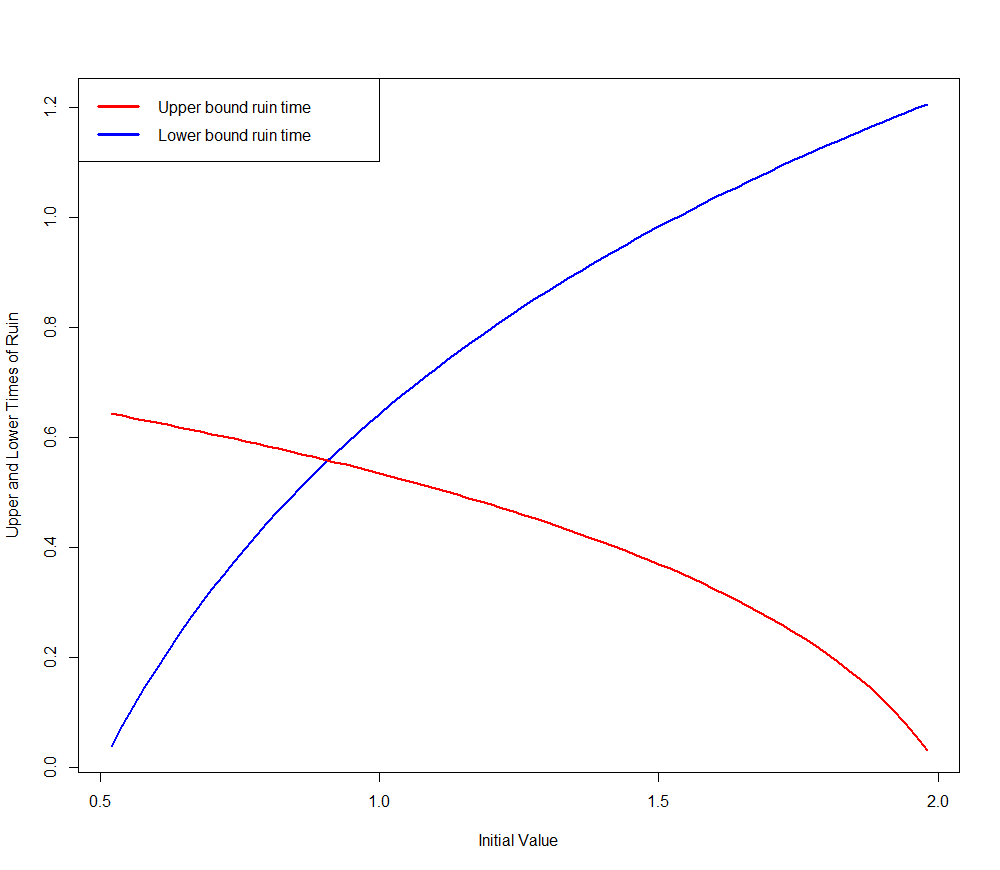}
\caption{The system under consideration is an OU process with $\mu = 1$, $\sigma=1$ and initial position $x_0$.
The red line is the ``optimal hitting time'' for the upper curve $u_0 e^{\alpha t}$ with $u_0=2$, $\alpha =1.3$,
i.e.\ the solution of (\ref{topt-u}). Note that this optimal time decreases to zero as $x_0$ increases
to $u_0=2$. Respectively, the blue line is the corresponding ''optimal hitting time'' for the lower curve $v_0 e^{\beta t}$,
$\beta=0.5$, $v_0=0.5$, i.e.\ the solution of (\ref{topt-v}). In this case the optimal time increases as the
distance of $x_0$ from $v_0$ increases.}
\label{fig:r-4}
\end{center}
\end{figure}
\begin{figure}[ht!]
\begin{center}
\includegraphics[width=3in]{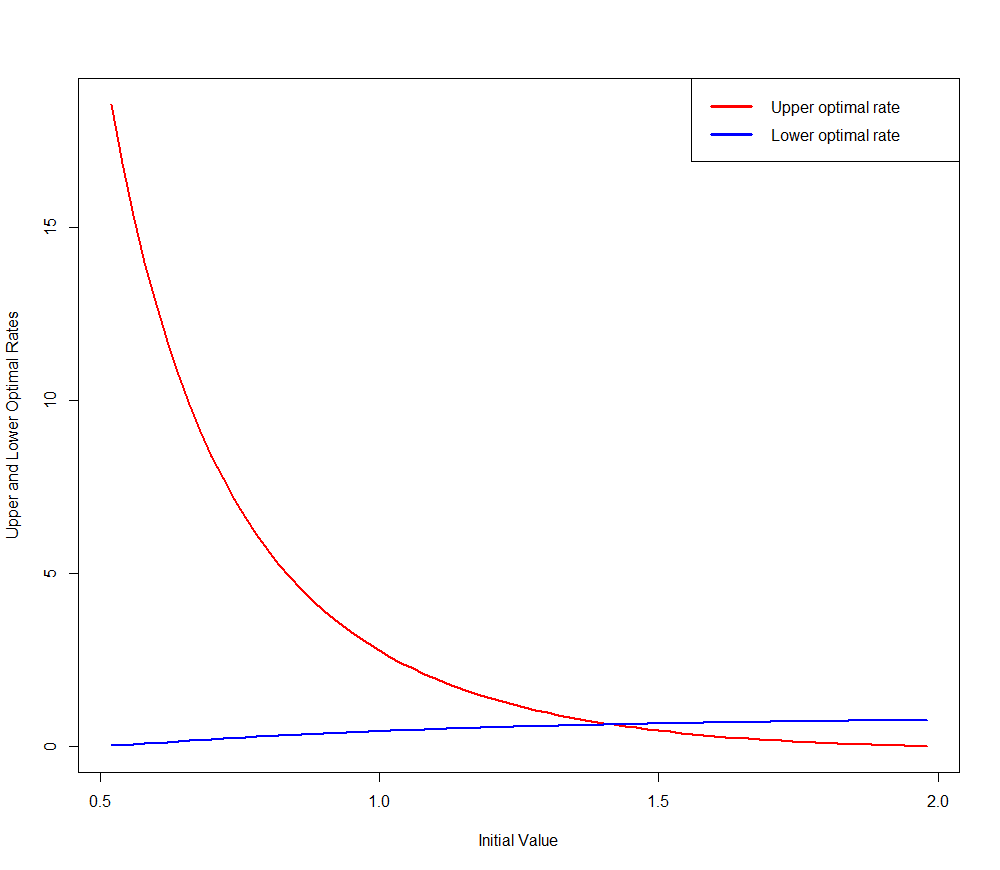}
\caption{The OU process and the upper and lower curves are as in Figure \ref{fig:r-4}. The red line is a 
	plot of the optimal rate $I_U$ for hitting the upper curve in the infinite 
	horizon problem given by (\ref{inf_ub_rate}). Correspondingly, the blue line gives the plot 
	of the optimal rate for hitting the 
	lower curve, $I_V$, given by (\ref{inf_lb_rate}). The point of intersection of the two curves corresponds 
	to the initial condition $x_0$ for which the exponential rate for the probability of hitting the upper
curve is equal to that for the lower curve.}
\label{fig:r-5}
\end{center}
\end{figure}

\section{More general models}
\subsection{Ornstein-Uhlenbeck with a general linear drift}

Here we consider the Ornstein-Uhlenbeck process with a more general drift. This is important since it arises as a
diffusion approximation in the risk models with interest rates considered in the Introduction. Consider the SDE
\[
dX_t = (\mu X_t +r )dt + \sigma dW_t, \;\;\; X_0=x_0.
\]
The upper boundary curve is $U(t)=u_0 e^{\alpha t}$. We assume that 
$u_0 > x_0$ and $\mu < \alpha$. In the deterministic limit, when 
$\sigma \rightarrow 0$, one obtains the Ordinary Differential Equation
$\frac{d}{dt}x(t) = \mu x(t) +r$ which has the solution $x(t) = x_0 e^{\mu t} + \frac{r}{\mu}(e^{\mu t} -1)$.
To ensure that we remain in range of applicability of Large Deviation results we will need to ensure that the 
deterministic solution remains strictly below the upper bound, $U(t)$ for all $t\geq 0$. Let
\begin{equation} \label{r-requation}
\phi(t) \; := \; U(t)-x(t) \;=\; u_0e^{\alpha t} - \left(x_0 + \frac{r}{\mu}\right)e^{\mu t} + \frac{r}{\mu}.
\end{equation}
Then we must have
\begin{equation}  \label{phi-positive}
\inf_{t\geq 0}  \phi(t) >0.
\end{equation}
We will make the additional assumption that
\begin{equation} \label{AdAss}
r \; < \; u_0 (\alpha -\mu).
\end{equation}
This is a sufficient condition that ensures that (\ref{phi-positive}) holds. Indeed, $\phi(0) = u_0 -x_0 >0$ and
\[
\phi'(t) = e^{\mu t} \left[ u_0\alpha e^{(\alpha - \mu)t} - x_0 \mu - r\right].
\]
Then,
$
 u_0\alpha e^{(\alpha - \mu)t} - x_0 \mu - r \geq u_0 \alpha - x_0 \mu - r >  u_0 \alpha - x_0 \alpha - r > 0
$
and hence (\ref{phi-positive}) holds.

The action functional is $\frac{1}{2\sigma^2} \int_0^T \left( x'-\mu x -r \right)^2 du .$
The Euler-Lagrange differential equation $F_x - \frac{d}{dt} F_{x'}=0$  reduces to
\[
x''-\mu^2 x - \mu r = 0
\]
with general solution 
\begin{equation} \label{gsgl}
x(t) = C_1 e^{\mu t} + C_2 e^{-\mu t} - \frac{r}{\mu} .
\end{equation}
The boundary conditions are
\begin{eqnarray}  \label{r-b1}
x_0 &=& C_1 + C_2 - \frac{r}{\mu} \\ \label{r-b2}
u_0 e^{\alpha T} &=& C_1 e^{\mu T} + C_2 e^{-\mu T} - \frac{r}{\mu}.
\end{eqnarray}
The transversality condition that must be satisfied by a critical path meeting the curve 
$U(t):=u_0 e^{\alpha t}$ at $T$ is
\[
F + (U'(T)-x'(T))F_{x'} =0 \;\; \mbox{ or } \;\; (x'-r-\mu x)\left(-x'-r -\mu x + 2 u_0\alpha e^{\alpha T} \right) =0
\]
which, using (\ref{gsgl}),  reduces to
\begin{equation} \label{intermediate-r}
C_2 \left( u_0 \alpha e^{\alpha T}-\mu C_1 e^{\mu T} \right) \;=\; 0.
\end{equation}
The above equation leads to the examination of two cases:
\paragraph{Case 1. $C_2=0$.} Using this value in (\ref{r-b1}), (\ref{r-b2}), and eliminating $C_1$ 
among them gives
\begin{equation} \label{case1-oudr}
u_0e^{\alpha T} - \left( x_0+ \frac{r}{\mu} \right) e^{\mu T} + \frac{r}{\mu} \;=\; 0.
\end{equation}
This equation corresponds to the requirement $\phi(T)=0$ for the function defined in 
(\ref{r-requation}) which is impossible. Hence  $C_2=0$ is impossible.

\paragraph{Case 2. $u_0 \alpha e^{\alpha T}-\mu C_1 e^{\mu T} =0$.}
This, together with (\ref{r-b2}) gives
\begin{equation} \label{intermediate-r1}
u_0\left( 1- \frac{\alpha}{\mu} \right) e^{\alpha T} = C_2 e^{-\mu T} - \frac{r}{\mu}. 
\end{equation}
Using this, (\ref{r-b1}), (\ref{r-b2}), give
\begin{eqnarray} \label{r-b-s1}
C_1 + C_2 &=& x_0 + \frac{r}{\mu} \\ \label{r-b-s2}
C_1  e^{\mu T} + C_2 e^{-\mu T} &=& u_0 e^{\alpha T} + \frac{r}{\mu} .
\end{eqnarray}
The above system has the solution
\[
C_1 \;=\; \frac{ e^{-\mu T} \left(x_0+ \frac{r}{\mu} \right) - \left( u_0 e^{\alpha T}
+ \frac{r}{\mu}\right)}{ e^{-\mu T} - e^{\mu T} }, \;\;\;\;
C_2 \;=\; \frac{ u_0 e^{\alpha T} + \frac{r}{\mu}
- e^{\mu T} \left(x_0+ \frac{r}{\mu} \right)}{ e^{-\mu T} - e^{\mu T} }.
\]
Using this, (\ref{intermediate-r1}) reduces to
\begin{equation} \label{r-T}
u_0 \left(\frac{\alpha}{\mu} -1 \right) e^{(\alpha+\mu)T} - u_0\frac{\alpha}{\mu} e^{(\alpha-\mu)T}
- \frac{r}{\mu} e^{\mu T} + x_0 + \frac{r}{\mu}  =0.
\end{equation}
Under Assumption (\ref{AdAss}) i.e.\ if the drift term $r$ is either negative or, if positive,
not too large, the above equation has a unique solution which determines $T$.

Define
$f(t) := u_0\left( \frac{\alpha}{\mu}- 1\right) e^{t(\alpha+\mu)} - u_0 \frac{\alpha}{\mu} e^{(\alpha- \mu)t}
- \frac{r}{\mu} e^{\mu t} + x_0 + \frac{r}{\mu}$ and note that $f(0) = x_0 -u_0 < 0$, 
$\lim_{t\rightarrow \infty} f(t) = + \infty$.
Also $f'(t) = (\alpha + \mu) u_0 \left( \frac{\alpha}{\mu} -1\right) e^{t(\alpha+\mu)}
- u_0 \frac{\alpha}{\mu} (\alpha -\mu) e^{(\alpha -\mu) t} - r e^{\mu t}$ and $f'(0) = u_0(\alpha -\mu) -r$.
Under Assumption (\ref{AdAss}) $f'(0) >0$. We will show that the assumptions implies in fact 
$f'(t)>0$ for all $t>0$. Let
$
e^{-\mu t} f'(t) =:g(t) = (\alpha+\mu)u_0 \left(\frac{\alpha}{\mu}-1\right) e^{\alpha t}
- u_0 \frac{\alpha}{\mu}(\alpha -\mu) e^{(\alpha -2 \mu )t} -r
$
and thus $g(0)= f'(0)= u_0(\alpha-\mu)-r >0$. Also
$
g'(t) = \frac{\alpha}{\mu}e^{\alpha t} (\alpha-\mu) u_0
\left[ \alpha + \mu - (\alpha -2 \mu) e^{-2\mu t} \right]  >0 \; \mbox{ for all } t\geq 0.
$
From the above we conclude that $g(t)$ and therefore $f'(t)$ is strictly positive for all $t$ 
as well and thus that (\ref{r-T}) has a unique solution.
The optimal path is then 
\begin{eqnarray}  \label{star-optimal-path_1}
x^*(t) &=&  \frac{\left(x_0+\frac{r}{\mu}\right) \sinh(\mu(T-t))
+ \left(u_0 e^{\alpha T}+\frac{r}{\mu} \right) \sinh(\mu t)}{\sinh(\mu T)} - \frac{r}{\mu}. \\  \nonumber
\end{eqnarray}
It is possible to show again that there are two intersction points between the curves $U(t)=u_0e^{\alpha t}$ 
and $x^*(t)$, one of which is of course $T_U$.
The condition for this solution to satisfy the inequality constraints as well is
\[
{x^*}'(T) > u_0 \alpha e^{\alpha T} .
\]
This condition is written as
\[
\mu
\frac{\left( u_0 e^{\alpha T}  + \frac{r}{\mu} \right)
\left( e^{\mu T} + e^{-\mu T} \right) - 2(x_0 + \frac{r}{\mu})}{e^{\mu T} - e^{-\mu T}}
> \alpha u_0 e^{\alpha T}
\]
and is is equivalent to
\[
\frac{r}{\mu} \left( e^{\mu T} + e^{-\mu T} \right)
+ u_0 e^{T(\alpha - \mu)} - 2 (x_0 + \frac{r}{\mu} ) > u_0 e^{\alpha T}
\left[ \left(\frac{\alpha}{\mu} - 1\right) e^{\mu T} - \frac{\alpha}{\mu} e^{- \mu T} \right]
\;=\; \frac{r}{\mu} ( e^{\mu T} - 1 ) - x_0
\]
the last equation following from (\ref{r-T}).
Hence the condition is equivalent to
\[
\frac{r}{\mu} e^{-\mu T} + u_0 e^{(\alpha -\mu)T} > x_0 + \frac{r}{\mu}.
\]
This inequality however is true because it is equivalent to $\phi(T)>0$ for the function
$\phi$ defined in (\ref{r-requation}), which is true.

The optimal rate can be obtained from the fact that $x'(t)-\mu x(t) - r = 2C_2 e^{\mu t}$ and hence
\[
I \;=\; \frac{\mu}{\sigma^2} \, \int_0^T 4 C_2^2 e^{\mu t} dt \;=\;
\frac{\mu}{\sigma^2} \,
\frac{\left( u_0 e^{(\alpha-\mu)T} - \frac{r}{\mu} \left( 1- e^{-\mu T} \right) - x_0  \right)^2}{1-e^{-2\mu T}}.
\]
Note, of course, that when $r\rightarrow 0$ the above reduces to the value of $I$ given in (\ref{alt-III}).

\subsection{A ruin problem involving two independent OU processes}
Here we generalize the problem examined in the previous section.
The lower (or upper) deterministic exponential boundary now is
also considered to be stochastic - in fact another, independent, OU process. 
We may thus study the following pair of SDE's
\begin{eqnarray}
dX_t &=& \alpha X_t dt + \sigma dW_t, \hspace{0.3in} X_0 = x_0   \label{SDE1}\\
dY_t &=& \beta Y_t dt + b dB_t, \hspace{0.3in} Y_0 =y_0 .\label{SDE2}
\end{eqnarray}
where $\beta < \alpha$ and $y_0 < x_0$. As a result of these inequalities,
in the absence of noise, ($\sigma=b=0$) we would have
$Y_t<X_t$ for all $t$. The presence of noise may cause the two curves to meet however.
Again, an exact analysis does not give results in
closed form and we obtain low noise logarithmic asymptotics in the Wentzell-Freidlin framework.
Using again Theorem 5.6.7, p.\ 214 of Dembo and Zeitouni (2010)
we obtain a two dimensional version of (\ref{Action}) for the action functional to be minimized:
\begin{equation} \label{N1}
I \;=\; \int_0^T F(x,x',y,y') dt, \hspace{0.3in} F\;=\;
\frac{1}{2} \left[ \frac{1}{\sigma^2} (x'-\alpha x)^2 + \frac{1}{b^2} (y'-\beta y)^2\right].
\end{equation}
The boundary conditions $x(0)=x_0$, $y(0)=y_0$, and $x(T)=y(T)$.

We will again tackle the infinite horizon problem directly and solve the moving boundary variational
problem using the appropriate transversality conditions. Thus the first order necessary conditions
for an extremum are
\begin{eqnarray}  \label{N-conditions-1}
&& F_x - \frac{d}{dt} F_{x'} = 0 , \hspace{0.2in}  F_y - \frac{d}{dt} F_{y'} = 0  \\ \label{N-conditions-2}
&& x(T)=y(T)  \\                      \label{N-conditions-4}
&& F_{x'} + F_{y'} = 0 \;\; \mbox{ at } T,\\ \label{N-conditions-5}
&& F - x'F_{x'} - y'F_{y'} =0 \;\; \mbox{ at } T.
\end{eqnarray}
The Euler-Lagrange equations (\ref{N-conditions-1}) give
$x''- \alpha^2 x = 0$ and $y'' - \beta^2 y = 0$
and thus,
$x(t) = C_1 e^{\alpha t} + C_2 e^{-\alpha t}$ and $y(t) = C_3 e^{\beta t} + C_4 e^{-\beta t}$
with boundary conditions
\begin{eqnarray}  \label{NN}
C_1+C_2 = x_0,  \hspace{0.2in} C_3+C_4 = y_0 , \hspace{0.1in} \mbox{and} \hspace{0.1in} C_1e^{\alpha T}
+ C_2 e^{-\alpha T} = C_3 e^{\beta T}+C_4 e^{-\beta T}.
\end{eqnarray}

The first transversality condition, (\ref{N-conditions-4}) gives
\begin{equation} \label{N3}
 \frac{1}{\sigma^2} (x'(T)-\alpha x(T)) + \frac{1}{b^2} (y'(T)-\beta y(T)) = 0
\end{equation}
or
\begin{equation} \label{N4}
\frac{\alpha}{\sigma^2}  C_2 e^{-\alpha T} + \frac{\beta}{b^2}  C_4 e^{-\beta T} = 0 .
\end{equation}
The second transversality condition (\ref{N-conditions-5}), after routine algebraic manipulations, gives
\begin{eqnarray*}
x'F_{x'}+y'F_{y'} -F  &=&   \frac{1}{2\sigma^2} (x'-\alpha x)(x'+\alpha x)
+ \frac{1}{2b^2} (y'-\beta y)(y'+\beta y) \;=\;0.
\end{eqnarray*}
The above, in view of (\ref{N3}), becomes
\[
\left(x'(T)-\alpha x(T)\right)\, \left(x'(T)+\alpha x(T) -y'(T) - \beta y(T) \right)=0 .
\]
If the first factor is zero then, in view of (\ref{N3}), we obtain
\[
x'(T)-\alpha x(T)=0,  \;\;\; y'(T)-\beta y(T)=0.
\]
In view of the fact that $x'(T)-\alpha x(T)= -2\alpha C_2 e^{-\alpha T}$ this translates into $C_2=0$ and 
similarly $y'(T)-\beta y(T)=-2\beta C_4 e^{-\beta T}=0$ implies $C_4=0$. Hence 
$x(t)=x_0e^{\alpha T}$, $y(t)=y_0e^{\beta T}$, and $x(T)=y(T)$ implies that 
$x_0 e^{\alpha T} = y_0e^{\beta T}$ or $e^{(\alpha-\beta)T} = \frac{y_0}{x_0}$. Since
$\alpha-\beta >0$ and $y_0/x_0<1$ it is impossible to find $T>0$ which satisfies this last equation.

The alternative solution is
\begin{equation} \label{N5}
x'(T)+\alpha x(T) \;=\; y'(T) + \beta y(T) .
\end{equation}
Note that
\[
x'(T)+\alpha x(T) = 2\alpha C_1 e^{\alpha T}, \;\;\; y'(T)+ \beta y(T) = 2\beta C_3 e^{\beta T}
\]
and hence (\ref{N5}) gives
\begin{equation} \label{N6}
\alpha C_1 e^{\alpha T} \;=\; \beta C_3 e^{\beta T}.
\end{equation}

\paragraph{\underline{Determination of the optimal path.}}
Displays (\ref{NN}), (\ref{N4}), and (\ref{N6}) provide the following five equations to determine the five unknown quantities,
$C_i$, $i=1,\ldots,4$, and $T$:
\begin{eqnarray}   \label{two-ind-1}
C_1+C_2 &=& x_0 \\  \label{two-ind-2}
\frac{\alpha}{\beta} e^{(\alpha-\beta)T}\,C_1 - \frac{\alpha}{\beta}\, \frac{b^2}{\sigma^2} e^{-(\alpha-\beta)T}\,C_2 &=& y_0 \\  \label{two-ind-3}
C_1e^{\alpha T} + C_2 e^{-\alpha T} &=& \frac{\alpha}{\beta} e^{\alpha T}\,C_1 - \frac{\alpha}{\beta}\, \frac{b^2}{\sigma^2} e^{-\alpha T}\,C_2 \\ \label{two-ind-4}
C_3 &=& \frac{\alpha}{\beta} e^{(\alpha-\beta)T}\,C_1  \\ \label{two-ind-5}
C_4 &=& - \frac{\alpha}{\beta}\, \frac{b^2}{\sigma^2} e^{-(\alpha-\beta)T}\,C_2
\end{eqnarray}
From the above we may obtain the values of $C_i$, $i=1,\ldots,4$ in terms of $T$:
\begin{eqnarray} \nonumber
C_1 &=& x_0 \, \frac{ \left(1+ \frac{\alpha}{\beta} \, \frac{b^2}{\sigma^2}\right) \, e^{-\alpha T}}{\left(1+ \frac{\alpha}{\beta} \, \frac{b^2}{\sigma^2}\right) \, e^{-\alpha T}+ \left(\frac{\alpha}{\beta}-1\right)\, e^{\alpha T}}, \;\;\;
C_2 \;=\; x_0 \, \frac{ \left(\frac{\alpha}{\beta}-1\right)\,e^{\alpha T}}{\left(1+ \frac{\alpha}{\beta} \, \frac{b^2}{\sigma^2}\right) \, e^{-\alpha T}+ \left(\frac{\alpha}{\beta}-1\right)\, e^{\alpha T}},\\
& &  \label{C12-second} \\ \nonumber
C_3 &=& y_0 \, \frac{ \left(1+ \frac{\beta}{\alpha} \, \frac{\sigma^2}{b^2}\right) \, e^{-\beta T}}{\left(1+ \frac{\beta}{\alpha} \, \frac{\sigma^2}{b^2}\right) \, e^{-\beta T}+ \left(\frac{\beta}{\alpha}-1\right)\, e^{\beta T}}, \;\;\;
C_4 \;=\; y_0 \, \frac{ \left(\frac{\beta}{\alpha}-1\right)\,e^{\beta T}}{\left(1+ \frac{\beta}{\alpha} \, \frac{\sigma^2}{b^2}\right) \, e^{-\beta T}+ \left(\frac{\beta}{\alpha}-1\right)\, e^{\beta T}}.
\end{eqnarray}
From these we obtain the following expression for the critical path
\begin{eqnarray}  \nonumber
x(t) &=& x_0 \, \frac{ \left(1+ \frac{\alpha}{\beta} \, \frac{b^2}{\sigma^2}\right) \, e^{\alpha(t-T)} \;+\;\left(\frac{\alpha}{\beta}-1\right)\,e^{\alpha(T-t)}}{\left(1+ \frac{\alpha}{\beta} \, \frac{b^2}{\sigma^2}\right) \, e^{-\alpha T}+ \left(\frac{\alpha}{\beta}-1\right)\, e^{\alpha T}} \, , \\
& &  \label{xt-second} \\ \nonumber
y(t) &=& y_0 \, \frac{ \left(1+ \frac{\beta}{\alpha} \, \frac{\sigma^2}{b^2}\right) \, e^{\beta(t-T)}\;+\; \left(\frac{\beta}{\alpha}-1\right)\,e^{\beta(T-t)}}{\left(1+ \frac{\beta}{\alpha} \, \frac{\sigma^2}{b^2}\right) \, e^{-\beta T}+ \left(\frac{\beta}{\alpha}-1\right)\, e^{\beta T}} \, .
\end{eqnarray}
Of course, there remains the task to determine the optimal meeting time $T$. From the above, when $t=T$
we have
\begin{eqnarray*} 
x(T) &=& x_0 \, \frac{\alpha(b^2+\sigma^2)}{(\alpha-\beta)\sigma^2 e^{\alpha T} + (\beta\sigma^2+\alpha b^2)e^{-\alpha T}}, \\ 
y(T) &=& y_0\, \frac{\beta(b^2+\sigma^2)}{(\beta-\alpha)b^2 e^{\beta T} + (\beta\sigma^2+\alpha b^2)e^{-\beta T}}.
\end{eqnarray*}
At the meeting time $T$, $x(T)=y(T)$ and therefore
\begin{equation} \label{T-for-two-OU}
x_0 \alpha \left[ (\beta-\alpha)b^2 e^{\beta T} + (\beta\sigma^2+\alpha b^2)e^{-\beta T} \right] \;=\; y_0 \beta \left[ (\alpha-\beta)\sigma^2 e^{\alpha T} + (\beta\sigma^2+\alpha b^2)e^{-\alpha T} \right].
\end{equation}

\paragraph{\underline{Determination of the meeting time $T$.}}
We will show that the above equation determines uniquely $T$. To this end, define the function
\[
f(t) \;:=\; (\alpha-\beta)\left[ y_0 \beta \sigma^2 e^{\alpha t} + x_0 \alpha b^2 e^{\beta t} \right] + (\beta \sigma^2 +\alpha b^2) \left[ y_0 \beta e^{-\alpha t} - x_0 \alpha e^{-\beta t} \right], \hspace{0.3in}
t \geq 0.
\]
It holds that
\begin{eqnarray*}
f(0) &=& (\alpha-\beta)\left[ y_0 \beta \sigma^2 + x_0 \alpha b^2 \right] + (\beta \sigma^2 +\alpha b^2) \left[ y_0 \beta - x_0 \alpha \right] \\
& = & \alpha \beta (\sigma^2+b^2)(y_0 \beta- x_0 \alpha) \;<\;0
\end{eqnarray*}
and also $\lim_{t\rightarrow \infty}f(t) =+\infty$. Furthermore
\[
f'(t) \;=\; (\alpha-\beta)\alpha\beta\left[ y_0  \sigma^2 e^{\alpha t} + x_0  b^2 e^{\beta t} \right] + (\beta \sigma^2 +\alpha b^2)\alpha\beta \left[ -y_0  e^{-\alpha t} + x_0 e^{-\beta t} \right]
\]
Clearly $f'(t)>0$ for all $t\geq 0$ since 
$\left[ -y_0  e^{-\alpha t} + x_0 e^{-\beta t} \right] = e^{-\alpha t} \left[ -y_0 + x_0 e^{(\alpha-\beta)t}\right]>0$ because $\alpha>\beta$ and $x_0 >y_0$.

\paragraph{\underline{$x(t)>y(t)$ when $t \in [0,T)$.}} A straight-forward computation (taking 
into account (\ref{xt-second}), (\ref{T-for-two-OU})) gives
\begin{equation} \label{xyprime}
x'(T) - y'(T) \;=\; - \; \frac{ x_0 \alpha(\alpha-\beta) (\sigma^2+b^2)}{(\alpha b^2 +\beta \sigma^2)e^{-\alpha T} + \sigma^2 (\alpha-\beta) e^{\alpha T}} \; < \;0
\end{equation}
Thus it can be seen that the path $x(\cdot)$ starts above $y(\cdot)$ at 0, crosses it from above at 
$T$ and (since $\alpha>\beta$) crosses it again once more
at some $T^* >T$. In particular we note that $x(t)> y(t)$ for $t \in [0,T)$, i.e. 
$x$ crosses $y$ at $T$ for the first time.

\paragraph{\underline{Determination of the rate $I$.}}
Taking into account that $x'(t)-\alpha x(t) = -2\alpha C_2 e^{-\alpha t}$ and similarly 
$y'(t)-\beta y(t) = -2 \beta C_4 e^{-\beta t}$ the rate function becomes
\begin{eqnarray*}
I &=& \frac{1}{2\sigma^2} \int_0^T 4 \alpha^2 C_2^2 e^{-2\alpha t}dt + \frac{1}{2b^2} \int_0^T 4 \beta^2 C_4^2 e^{-2\beta t}dt \;=\;
\frac{\alpha C_2^2}{\sigma^2} \left( 1-e^{-2\alpha T}\right)
+ \frac{\beta C_4^2}{b^2} \left( 1-e^{-2\beta T}\right) \\
&=&  \frac{ \frac{\alpha}{\sigma^2} \left( 1-e^{-2\alpha T}\right)  x_0^2 \,\left(\frac{\alpha}{\beta}-1\right)^2\,e^{2\alpha T}}
{\left[\left(1+ \frac{\alpha}{\beta} \, \frac{b^2}{\sigma^2}\right) \, e^{-\alpha T}+ \left(\frac{\alpha}{\beta}-1\right)\, e^{\alpha T}\right]^2} \;+\;
\frac{\frac{\beta}{b^2} \left( 1-e^{-2\beta T}\right)  y_0^2 \left(\frac{\beta}{\alpha}-1\right)^2\,e^{2\beta T}}
{\left[\left(1+ \frac{\beta}{\alpha} \, \frac{\sigma^2}{b^2}\right) \, e^{-\beta T}+ \left(\frac{\beta}{\alpha}-1\right)\, e^{\beta T}\right]^2}
\end{eqnarray*}
or equivalently
\begin{equation} \label{I-two-OU}
I \;=\; \frac{\alpha (\alpha-\beta)^2\sigma^2 x_0^2\left(e^{2\alpha T}-1\right)}
{\left[(\alpha-\beta)\sigma^2 e^{\alpha T} + (\beta\sigma^2+\alpha b^2)e^{-\alpha T}\right]^2} \;+\;
\frac{\beta \left(\alpha -\beta\right)^2  b^2 y_0^2 \left( e^{2\beta T}-1\right) }
{ \left[\left(\beta - \alpha \right)b^2\, e^{\beta T} +\left(\beta \sigma^2 + \alpha b^2  \right) \, e^{-\beta T} \right]^2} .
\end{equation}

In particular, when $b=0$ and $\alpha=\mu$ then the lower OU process 
becomes a deterministic lower bound and (\ref{I-two-OU})
reduces indeed to the right hand side of (\ref{inf_ub_rate}), as it should.

Again, as in the proof of Theorem \ref{th:OU-1} we will show that the solution 
obtained corresponds to a global minimum using the fact that
$F:\bR^4 \rightarrow \bR$ is convex and appealing to 
Theorem 3.16 of Brechtken-Manderscheid (1994), p.\ 45.  
To establish the convexity of $F(x,x',y,y') := \frac{1}{2\sigma^2}(x'-\alpha x)^2 + \frac{1}{2 b^2}(y'-\beta y)^2$
we note that, for any $(x_0,x_0',y_0,y_0') \in \mathbb{R}^4$,
\begin{equation} \label{convex-F}
F(x,x',y,y') - F(x_0,x'_0,y_0,y'_0) \geq F_x^0 \, (x-x_0) + F_{x'}^0 \,(x'-x'_0) + F_y^0 \, (y-y_0) + F_{y'}^0 \,(y'-y'_0)
\end{equation}
where $F^0_x$ is shorthand for $F_x(x_0,x'_0,y_0,y_0')$ and similarly for the other three 
such quantities. The above inequality is equivalent to
\begin{eqnarray*}
&& \frac{1}{2\sigma^2}\left(x'-\alpha x\right)^2 + \frac{1}{2b^2}\left(x'-\beta x\right)^2 - \frac{1}{2\sigma^2}\left(x'_0-\alpha x_0\right)^2
- \frac{1}{2b^2}\left(x'_0-\beta x_0\right)^2   \\
&& \hspace{1in} \; \geq \;  - \frac{\alpha}{\sigma^2}\left(x_0'-\alpha x_0\right)\, (x-x_0) + \frac{1}{\sigma^2} \left(x_0'-\alpha x_0\right)\, (x'-x'_0) \\
&& \hspace{1in} \;\;\; \;  - \frac{\beta}{b^2}\left(y_0'-\beta y_0\right)\, (y-y_0) + \frac{1}{b^2} \left(y_0'-\beta y_0\right)\, (y'-y'_0) .
\end{eqnarray*}
Elementary algebraic manipulations can show the above inequality to be true and therefore 
establish inequality (\ref{convex-F}) which implies the convexity of $F$.

We may summarize the above long derivation in the following
\begin{theorem}
Consider the pair of Ornstein-Uhlenbeck SDE's depending on a parameter $\epsilon >0$
\begin{eqnarray*}
dX_t^\epsilon &=& \alpha X_t^\epsilon dt + \sqrt{\epsilon} \sigma dW_t, \hspace{0.3in} X_0^\epsilon = x_0,  \\
dY_t^\epsilon &=& \beta Y_t^\epsilon dt + \sqrt{\epsilon} b dB_t, \hspace{0.3in} Y_0^\epsilon =y_0 . %
\end{eqnarray*}
Assume that $0<y_0<x_0$ and $0<\beta < \alpha$. Let $T^\epsilon:=\inf\{t\geq 0: X_t^\epsilon = Y_t^\epsilon\}$ 
(with the standard convention that $T^\epsilon=+\infty$ if the set is empty). Then
\[
\lim_{\epsilon \rightarrow 0} \epsilon \log \bP(T^\epsilon < \infty) \;=\; -I
\]
where $I$ is given by (\ref{I-two-OU}). If this rare event occurs then the meeting path 
followed by the two processes is given by (\ref{xt-second}) 
and the meeting time $T$ is the unique solution of (\ref{T-for-two-OU}).
\end{theorem}

\section{Geometric Brownian Motion} \label{sec:GBM}

In this section an analysis of the problems we examined for the Ornstein-Uhlenbeck process 
is repeated for the Geometric Brownian motion. The
approach followed and the techniques used are analogous to those of section 2. 
The reason for treating the Geometric Brownian motion
in some detail here is due to its great importance in applications 
but also to the fact that in this case an analytic solution in closed form for the types
of ruin problems we consider can be obtained. As a result, the accuracy 
and merit of the large deviation estimates we obtain may be gauged.

\subsection{The finite horizon problem}
Suppose that $\{X_t;t\geq 0\}$ is a Geometric Brownian motion satisfying 
the Stochastic Differential Equation
\begin{equation} \label{GBM-th1}
dX_t = \mu X_t dt + \sigma X_t dW_t, \hspace{0.3in} X_0 = x_0 \; \mbox{w.p. 1.}
\end{equation}
As is well known this has the closed form solution
\begin{equation} \label{GBM-th2}
X_t = x_0 e^{(\mu-\frac{1}{2}\sigma^2)t + \sigma W_t}.
\end{equation}
Let $u_0 >x_0$ and $\alpha>\mu$. Then the event $\{X_t \geq u_0 e^{at} \mbox{ for some } t\leq T\}$ 
is an event whose probability goes to 0 as $\sigma \rightarrow 0$. Our goal is to obtain low 
variance Wentzell-Freidlin asymptotics for this finite horizon hitting probability. For reasons of 
notational compatibility we introduce the {\em parametrized process}
\begin{equation} \label{GBM-the3}
dX_t^\epsilon = \mu X_t^\epsilon dt + \sqrt{\epsilon} \sigma X_t^\epsilon dW_t, \hspace{0.3in} X_0^\epsilon 
= x_0 \; \mbox{w.p. 1.}
\end{equation}

\begin{theorem}   \label{Th:GBM1}
For the parametrized process $\{X_t^\epsilon\}$,
\begin{equation} \label{GBM-th4}
\lim_{\epsilon \rightarrow 0} \epsilon 
\log \mathbb{P}\left(\sup_{0\leq t \leq T}\left(X_t^\epsilon - u_0 e^{\alpha t} \right) \geq 0 \right) \;=\; - I(T) .
\end{equation}
The {\em rate function} $I(T)$ is given by
\begin{equation}   \label{GBM-th45}
I(T) := \min_{0\leq t \leq T} J_*(t)
\end{equation}
where $J_*(t)$ is solution to the minimization problem
\begin{equation} \label{GBM-th5}
J_*(t) \;=\;
\min\left\{ J(x,t)\;:\; x \in \mathcal{H},\; x(0)=x_0, \; x(t)=u_0e^{\alpha t}, \; x(s) < u_0 e^{\alpha s}, \;
s \in [0,t) \right\}.
\end{equation}
In the above, $\mathcal{H}$ is the Cameron-Martin space of absolutely continuous functions 
with square integrable derivative (as defined in the statement of Theorem \ref{th:Schilder}) 
and $J(x,t)$ is the {\em action functional}
\begin{equation} \label{GBM-th6}
J(x,t) := \frac{1}{2} \int_0^t \left( \frac{x'(s) - \mu x(s)}{\sigma x(s)} \right)^2 ds
= \frac{1}{2\sigma^2}  \int_0^t \left( (\log x(s))'-\mu \right)^2 ds.
\end{equation}
\end{theorem}
The proof of this theorem is along the lines of the proof of Theorem \ref{th:OU-1}, but a great deal easier. 
Instead of a proof we give a brief discussion. The minimizing path $x(s)$ can be easily obtained in this case
either by using the the Euler-Lagrange differential equation, or simply by observing that the functional $J(x,t)$ 
is minimized when $(\log x)'$ is constant, say $c$, or equivalently when $\log x(s) =B+ cs$ for some $B \in \bR$ 
and $s \in [0,t]$. This in turn implies that $x(s) = K e^{cs}$ with $x(0)=x_0 = K$ and
$x(t) = x_0 e^{ct} = u_0 e^{\alpha t}$ whence we conclude that the function that minimizes 
the action functional under the boundary conditions is
\begin{equation} \label{GBM-th7}
x(t) = x_0 e^{ct} \hspace{0.2in} \mbox{where} \;\;\; c = \alpha + \frac{1}{t} \log \frac{u_0}{x_0}.
\end{equation}
It is easy to see that the above path satisfies the constraint $x(s) < u_0 e^{\alpha s}$ for $s \in [0,t)$. 
The corresponding minimum action is then
\[
J_*(t) \;=\; \frac{t}{2\sigma^2} \left( \alpha - \mu + \frac{1}{t} \log \frac{u_0}{x_0}  \right)^2
\]
or
\[
J_*(t) \;=\; t\,\frac{(\alpha -\mu)^2}{2\sigma^2} + 2 \frac{(\alpha-\mu)  \log \frac{u_0}{x_0} }{2\sigma^2}
+ \frac{1}{t} \, \frac{(\log \frac{u_0}{x_0})^2}{2\sigma^2}.
\]
The above function is clearly convex on $(0,\infty)$ and achieves its minimum at 
$t_{\text{min}} = \frac{ \log \frac{u_0}{x_0}}{\alpha -\mu}$. This minimum is
$J_*(t_{\text{min}}) = \frac{ 2}{\sigma^2} (\alpha -\mu) \log\frac{u_0}{x_0}$.
Thus, since $J_*$ is easily seen to be decreasing in $(0,t_{\text{min}})$ and increasing in $(t_{\text{min}},\infty)$,
by (\ref{GBM-th45}) we conclude that the rate function is
\begin{equation} \label{GBM-th8}
I(T) = \left\{ \begin{array}{lll} 
\frac{T}{2\sigma^2} \left( \alpha - \mu + \frac{1}{T} \log \frac{u_0}{x_0}  \right)^2 & \mbox{if} & T <  t_{\text{min}}  \\ 
& & \\ 
\frac{ 2}{\sigma^2} (\alpha -\mu) \log\frac{u_0}{x_0} & \mbox{if}  & T \geq t_{\text{min}}  
\end{array} \right. \,.
\end{equation}
In practice, (\ref{GBM-th4}) gives rise to the approximation 
\begin{equation}  \label{GBM-th9}
-\log \mathbb{P}\left(\sup_{0\leq t \leq T} (X_t - u_0e^{at}) \geq 0 \right) \approx I(T)
\end{equation}
{\em which is satisfactory provided that $\sigma$ is sufficiently small.} We assess its quality in the next
subsection taking advantage of the fact that an exact, closed form solution also exists in this situation.

\subsection{The exact solution}
Consider the GBM $X_t = x_0 e^{(\mu - \frac{1}{2}\sigma^2)t + \sigma W_t}$ and the corresponding 
finite horizon hitting probability
\[
p_T \; :=\; \bP\left(\sup_{0 \leq t\leq T} (X_t - u_0 e^{a t}) \geq 0\right)
\]
where, as before $\alpha > \mu$ and $0 <x_0< u_0$. Since the event $(X_t - u_0 e^{\alpha t}) \geq 0$ 
is the same as $X_t e^{-\alpha t} -u_0 \geq 0$,
\begin{eqnarray} \nonumber
p_T &=& \bP\left(\sup_{0 \leq t\leq T}  x_0 e^{(\mu - \frac{1}{2}\sigma^2-a)t + \sigma W_t} \geq u_0 \right) \;=\;
\bP\left( \sup_{0 \leq t\leq T} (\mu - \frac{1}{2}\sigma^2-\alpha)t + \sigma W_t \geq \log\frac{u_0}{x_0} \right) \\ 
\label{the-exact-solution}
&=& 1-  \Phi\left( \textstyle \frac{\log\left(\frac{u_0}{x_0}\right) - (\mu-\alpha -\frac{1}{2}\sigma^2)T}{\sigma \sqrt{T}} \right)  
+\, e^{\frac{2}{\sigma^2} (\mu-\alpha -\frac{1}{2}\sigma^2)\log\left(\frac{u_0}{x_0}\right)}
  \Phi\left( \textstyle \frac{-\log\left(\frac{u_0}{x_0}\right) - (\mu-\alpha -\frac{1}{2}\sigma^2)T}{\sigma \sqrt{T}} \right) .
\end{eqnarray}
Here, $\Phi(x):=\int_{-\infty}^x \frac{1}{\sqrt{2\pi}} \, e^{-\frac{1}{2}u^2} \,du$, the standard normal distribution function.
The above exact formula for $p_T$ allows us to evaluate the accuracy of the approximation (\ref{GBM-th9}).
Figure \ref{fig:GBM-3} shows again $- \sigma^2 \log p_T$ together with the Wentzell-Freidlin
asymptotic result when $\sigma \rightarrow 0$. Thus approximation (\ref{GBM-th9}) may be considered satisfactory, 
provided that $\sigma$ is small.

\subsection{The infinite horizon problem}
The exact value of the infinite horizon hitting probability can be obtained from 
(\ref{the-exact-solution}) by letting $T\rightarrow \infty$. This gives
\begin{eqnarray*}
\lim_{T\rightarrow \infty} p_T &=:&  p_\infty \;=\;
\exp\left( \frac{2}{\sigma^2} \left(\mu - \frac{1}{2}\sigma^2-\alpha \right) \log\frac{u_0}{x_0} \right).  
\end{eqnarray*}
Returning to the parametrized version of the problem, concerning the family of processes 
$\{X_t^\epsilon\}$ defined in (\ref{GBM-the3}), the corresponding infinite horizon hitting probability is
\[
p_\infty^\epsilon \;=\; \exp\left( \frac{2}{\epsilon \sigma^2} \left(\mu - \frac{1}{2}\epsilon \sigma^2-\alpha \right)
\log\frac{u_0}{x_0} \right)
\]
and therefore
\begin{equation} \label{GBM-infinite-horizon}
\lim_{\epsilon\rightarrow 0} \epsilon \log p_\infty^\epsilon \;=\; -\frac{2}{\sigma^2}(\alpha-\mu)\log\frac{u_0}{x_0}.
\end{equation}
This, as we will see, is the same as the result obtained from Wentzell-Freidlin theory.
\begin{theorem}  \label{Th:GBM2}
For the parametrized process $\{X_t^\epsilon\}$
\begin{equation} \label{GBM-th4b}
\lim_{\epsilon \rightarrow 0} \epsilon \log 
\mathbb{P}\left(\sup_{ t \geq 0}\left(X_t^\epsilon - u_0 e^{\alpha t} \right) \geq 0 \right) \;=\; - I(\infty)
\end{equation}
where the {\em rate function} $I(\infty)$ is the solution to the infinite horizon variational problem
\begin{equation}  \label{GBM-th10}
\inf\left\{ J(x,T) \; : \;
x\in \mathcal{H},\; x(s)<u_0e^{\alpha s} \mbox{ for }  0\leq s < T,\; x(0)=x_0,\; x(T)=u_0e^{\alpha T} \right\}
\end{equation}
where $J(x,t):=\frac{1}{2} \int_0^t \left(\left( \log x(u) \right)' - \mu \right)^2 du$ and $\mathcal{H}$ 
is again the Cameron-Martin
space of absolutely continuous functions with square-integrable derivatives. In fact, 
for the infinite horizon problem, 
\begin{equation} \label{GBM-rate}
I(\infty) = 2 \frac{\alpha - \mu}{\sigma^2} \, \log \frac{u_0}{x_0},
\end{equation}
the optimal time horizon is
\begin{equation} \label{Trans-4}
T = \frac{ \log \frac{u_0}{x_0} }{ \alpha - \mu},
\end{equation}
and the {\em optimal path that achieves the minimum} is
\begin{equation} \label{GBM-Optimal-Path}
x^*(t) = x_0 e^{{2\alpha - \mu}t}, \;\;\; t\in[0,T].
\end{equation}
\end{theorem}
\begin{center}
\begin{figure}[!hb]
\begin{center}
\includegraphics[width=6.0in]{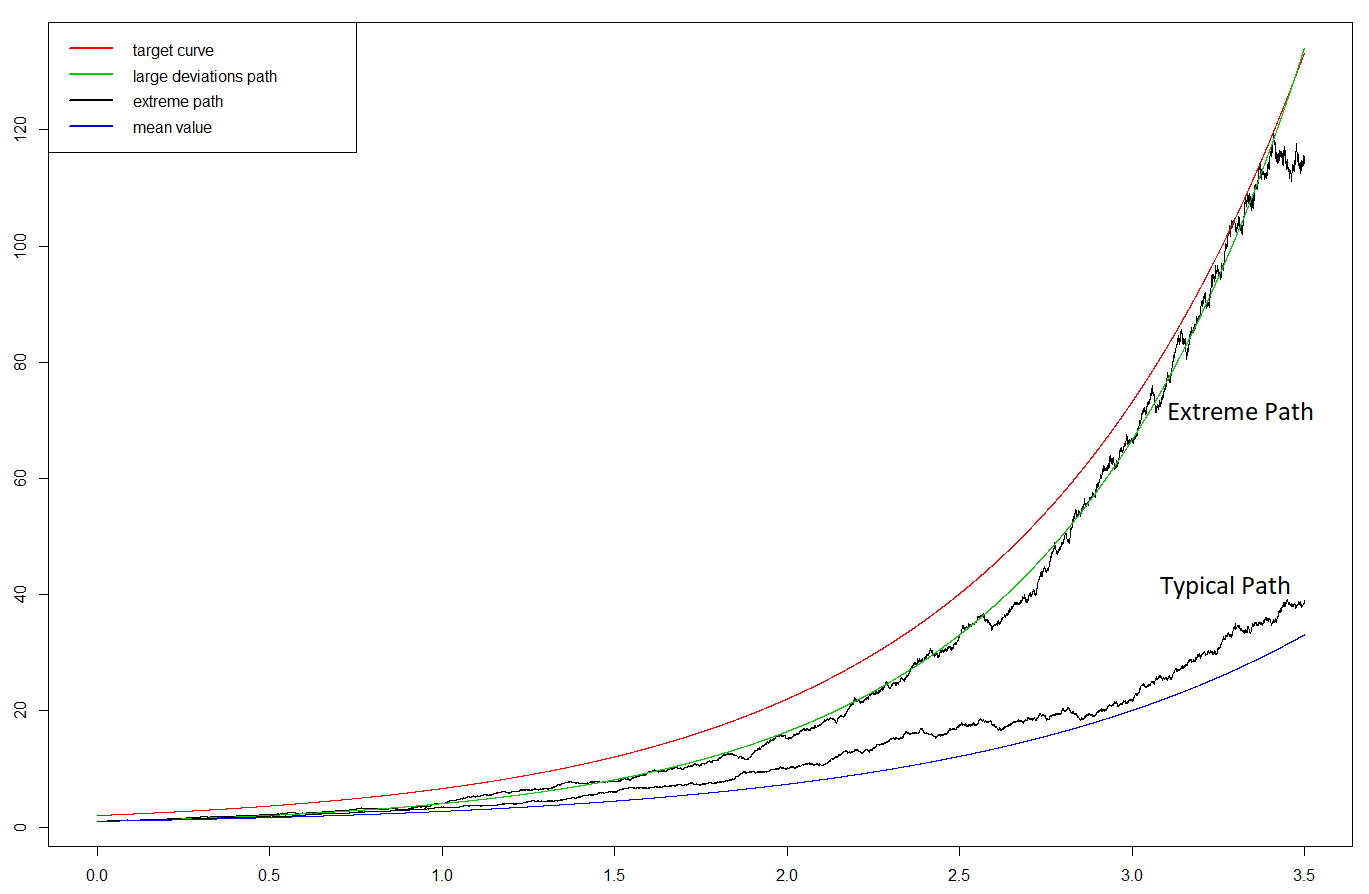}
\caption{Simulated sample path for $\alpha=1, x_0=1, u_0 =2$  and  $\sigma =0.15$. The red curve is the exponential 
target curve $u_0e^{\alpha t}$. The green curve is the optimal path  $x^*(t) = x_0 e^{(2\alpha - \mu)t}t$ given 
in (\ref{GBM-Optimal-Path}). Both a typical path and an extreme path of the Geometric Brownian motion are displayed. 
The extreme path was generated by simulating a large number of paths ($\approx 10^5$) and selecting one that 
hit the target, i.e.\ reached the red curve. As expected it follows closely the green curve. The smaller the variance the 
smaller the probability of hitting the target and the closer the agreement with the theoretical path.}
\label{GBM-optimal-path}
\end{center}
\end{figure}
\end{center}
\begin{proof}
The optimization problem of Theorem \ref{Th:GBM2} can of course be solved using the finite horizon analysis as a basis.
However we prefer to use standard techniques of the calculus of variations for infinite horizon problems with the final value of
the path constrained to lie on a prescribed curve using the {\em transversality conditions}
\begin{eqnarray}  \nonumber
&& \min \int_0^T F(x,x',t)dt, \hspace{0.2in} \mbox{ with boundary conditions } x(0)=x_0, \mbox{ and } x(T)=u(T) \\   \label{GBM-io1}
&& \mbox{with } F(x,x',t) = \frac{1}{2\sigma^2} \left( \frac{x'}{x} - \mu \right)^2.
\end{eqnarray}
In the above $u(t)=u_0 e^{\alpha t}$ is a given boundary curve with $x_0<u_0$ and $x$ is a $C^1[0,\infty)$ function which minimizes the
``action'' integral given the boundary conditions in (\ref{GBM-io1}). The conditions for a minimum is
\begin{eqnarray}
&& F_x - \frac{d}{dt} F_{x'} = 0 \label{GBM-EL}  \label{Euler-Lagrange}\\
&& x(0)=x_0 \;\;\; \mbox{ and } \;\;\; x(T)=u(T) \label{GBM-B1}   \\
&& F + (u'-x')F_{x'} =0 \;\; \mbox{ at } T. \label{Transversality-io}
\end{eqnarray}
The first equation is the Euler-Lagrange DE of the Calculus of Variations. 
Equation (\ref{Transversality-io}) is known as the {\em transversality condition}
resulting from the fact that the end time $T$ is not fixed but is itself to be 
chosen optimally, under the restriction that $x(T)=u(T)$. Then
the Euler-Lagrange equation (\ref{Euler-Lagrange}) becomes
\[
\frac{2}{x^3} \left( (x')^2 -  x'' x \right) = 0
\]
or equivalently
\[
\frac{x'}{x} = \frac{x''}{x'} \;\; \Leftrightarrow \;\; \left(\log x' \right)' - \left(\log x \right)' =0
 \;\; \Leftrightarrow \;\; \log x' - \log x = c_1  \;\; \Leftrightarrow \;\; \frac{x'}{x} = \gamma.
\]
Hence
\begin{equation} \label{GBMxt}
x(t) = x_0 e^{\gamma t}.
\end{equation}
The transversality condition (\ref{Transversality-io}) reduces to
\[
\left( \frac{x'(T)}{x(T)} - \mu \right) \, \left( \frac{x'(T)}{x(T)} - \mu  
+ \left( u_0 \alpha e^{\alpha T} - x'(T) \right)\frac{2}{x(T)}  \right) =0
\]
and taking into account (\ref{GBMxt}) we obtain either $\mu = \gamma$ or
\[
\gamma - \mu + 2 \alpha \frac{u_0}{x_0} e^{(\alpha-\gamma)T} - 2 \gamma =0
\]
or
\begin{equation} \label{Trans-1}
2 \alpha \frac{u_0}{x_0} e^{(\alpha-\gamma)T} = \mu + \gamma .
\end{equation}
Equation (\ref{GBM-B1}) gives $x_0 e^{\gamma T} = u_0 e^{\alpha T}$ and therefore
\begin{equation} \label{Trans-2}
e^{(\alpha-\gamma)T} = \frac{x_0}{u_0}.
\end{equation}
From (\ref{Trans-1}) and (\ref{Trans-2}) we have
\begin{equation} \label{Trans-3}
\gamma = 2\alpha - \mu
\end{equation}
The solution of the variational process that
minimizes the action functional $I$ and satisfies the boundary conditions yields the optimal path
$x_t=x_0e^{(2 \alpha -\mu) t}$ and the rate function
\[
I = 2 \frac{\alpha - \mu}{\sigma^2} \, \log \frac{u_0}{x_0} \;\; \mbox{ and } \;\; T 
= \frac{ \log \frac{u_0}{x_0} }{ \alpha - \mu}.
\]
\end{proof}
The exact solution agrees with the Wentzell-Freidlin asymptotic result. 
Figure \ref{GBM-optimal-path} illustrates the above result. In particular, the
extreme path was selected by simulating a large number of paths and picking the largest among them.
\subsection{Two correlated Geometric Brownian Motions} \label{sec:CGBM}

Suppose that $W_t$, $V_t$, are independent standard Brownian motions and $\rho \in [-1,1]$. Set
$B_t = \rho W_t + \sqrt{1-\rho^2} \, V_t$. Then $(W_t,B_t)$ are correlated Brownian motions with
correlation $\rho$. Consider now the processes
\begin{eqnarray*}
dX_t &=& \alpha X_t dt + \sigma X_t dW_t, \hspace{0.1in} X_0 = x_0,\\
dY_t &=& \beta Y_t dt + b Y_t dB_t, \hspace{0.1in} Y_0 = y_0.
\end{eqnarray*}
We will assume that $\alpha > \beta$ and $x_0 > y_0 >0$. Thus, in the absence of noise one would
have $X_t > Y_t$ for all $t>0$. In the presence of noise however the probability that $X_T = Y_T$
for some $T>0$ is non-zero. The second equation can be written equivalently as
\[
 dY_t = \beta Y_t dt + \rho b Y_t dW_t + \sqrt{1-\rho^2} b Y_t dV_t.
\]
Using once more Theorem 5.6.7, p.\ 214  of Dembo and Zeitouni (2010) we obtain again a two dimensional
version of (\ref{Action}) for the action functional to be minimized:
\begin{equation} \label{rate-correlatec}
I = \frac{1}{2} \int_0^T \left( \frac{x'-\alpha x}{x\sigma} \right)^2 + \frac{1}{1-\rho^2} \left( \frac{y'-\beta y}{y b}
- \rho \frac{x'-\alpha x}{x\sigma} \right)^2 dt
\end{equation}
This of course can be justified by appealing to the multidimensional version of (\ref{Action}) as we have already seen.
Set
\begin{equation} \label{F-function}
F = \frac{1}{2\sigma^2} \left( \frac{x'}{x} - \alpha \right)^2
+ \frac{1}{2(1-\rho^2)} \left( \frac{1}{b} \left(\frac{y'}{y}- \beta \right)- \frac{\rho}{\sigma} 
\left( \frac{x'}{x}-\alpha \right) \right)^2
\end{equation}
The conditions for minimum are
\begin{eqnarray}  \label{C1}
&& F_x - \frac{d}{dt} F_{x'} = 0  \\ \label{C2}
&& F_y - \frac{d}{dt} F_{y'} = 0  \\  \label{C3}
&& x(T)=y(T)  \\                      \label{C4}
&& F_{x'} + F_{y'} = 0 \;\; \mbox{ at } T,\\ \label{C5}
&& F - x'F_{x'} - y'F_{y'} =0 \;\; \mbox{ at } T.
\end{eqnarray}
Then, after some routine algebraic operations, (\ref{C1}) becomes
\[
\frac{1}{x} \left[ b^2\left( \frac{x''}{x} - \left(\frac{x'}{x}\right)^2 \right) 
- \rho b \sigma \left( \frac{y''}{y} - \left(\frac{y'}{y}\right)^2 \right)\right]=0
\]
which gives $b^2 (\log x)'' -\rho b \sigma (\log y)'' =0$. Similarly (\ref{C2}) gives
$\sigma^2 (\log y)'' -\rho b \sigma (\log x)'' =0$.
These equations together imply that $(\log x)'' = (\log y)'' =0$ whence we obtain $\frac{x'}{x} 
= c_1 $ and $\frac{y'}{y} = c_2$ for arbitrary $c_1$, $c_2$, and hence
\begin{equation} \label{sol1}
x(t) = x_0 e^{c_1 t}, \hspace{0.5in} y(t) = y_0 e^{c_2 t}.
\end{equation}
Condition (\ref{C3}) gives
\begin{equation} \label{sol2}
x_0 e^{c_1T} = y_0 e^{c_2T}.
\end{equation}
Taking into account that $\frac{x'}{x}=c_1$ and similarly $\frac{y'}{y}=c_2$, condition (\ref{C4}) gives
\[
\frac{1}{x_0e^{c_1T}} \left[ b^2 \left( c_1-\alpha\right) - \rho b \sigma \left(c_2 -\beta \right) \right] + \frac{1}{y_0 e^{c_2T}}
\left[ \sigma^2 ( c_2-\beta) - \rho b \sigma (c_1-\alpha) \right] =0.
\]
Setting 
\begin{equation} \label{sol3.5} 
u_1 = c_1-\alpha, \hspace{0.1in} u_2 = c_2 -\beta,
\end{equation}
and taking into account (\ref{sol2}) the above equation implies
$b^2 u_1 - \rho b \sigma u_2 + \sigma^2 u_2 - \rho b \sigma u_1 =0$.
This gives
\begin{equation} \label{sol3}
u_2 = \lambda u_1 \;\;\; \mbox{with} \;\;\; \lambda = \frac{b}{\sigma} \frac{\rho \sigma -b}{\sigma -\rho b}
\end{equation}
Finally, from (\ref{C5}),
\[
b^2 u_1^2 + \sigma^2 u_2^2 - 2 \rho b \sigma u_1 u_2 - 2 c_1 \left[b^2 u_1 - \rho b \sigma u_2\right] - 2 c_2 \left[ \sigma^2 u_2 - \rho b \sigma u_1 \right] = 0
\]
or, taking into account (\ref{sol3}) and (\ref{sol3.5})
\[
-u_1^2 \left[ b^2 + \sigma^2 \lambda^2 - 2 \rho b \sigma \lambda \right] + 2 u_1 \left[ -\alpha b^2 + \beta \rho b \sigma - \lambda \beta \sigma^2 + \lambda \alpha b \sigma \rho\right] =0.
\]
Besides the solution $u_1 = 0$ which means ($c_1 = \alpha$), we obtain
\[
u_1 = - \frac{2}{b^2+ \sigma^2 \lambda^2 - 2 \rho b\sigma \lambda} \left( \alpha b^2 + a \lambda \sigma^2 - \rho b \sigma (a+ \lambda \alpha) \right)
\]
and a corresponding expression for $u_2$. 
After routine algebraic manipulations we obtain 
\footnote{Note that, if $\sigma <b$ then 
$\rho \in(-1,\frac{\sigma}{b}) \Rightarrow u_1>0,\, u_2<0$ whereas $\rho \in (\frac{\sigma}{b} ,1) \Rightarrow
u_1<0,\,u_2<0$.\\
If $\sigma>b$ then $\rho \in(-1,\frac{b}{\sigma}) \Rightarrow u_1>0,\, u_2<0$ whereas $\rho \in (\frac{\sigma}{b} ,1) \Rightarrow
u_1>0,\,u_2>0$.\\
If $\sigma=b$ then $u_1>0,\, u_2<0$ for all values of $\rho$.}
\begin{equation} \label{sol4}
u_1 = 2(\beta-\alpha) \frac{\sigma(\sigma-\rho b)}{\sigma^2 + b^2 -2\rho b \sigma}, \;\;\;\;\;\;\;
u_2 = 2(\beta-\alpha) \frac{b(\rho\sigma- b)}{\sigma^2 + b^2 -2\rho b \sigma} .
\end{equation}
From (\ref{F-function}) and (\ref{sol4}), together with the definition of $u_1$, $u_2$,
\begin{eqnarray}   \label{F-function-2}
F &=& \frac{1}{2b^2\sigma^2(1-\rho^2)} \left[ b^2u_1^2 + \sigma^2 u_2^2 -2\rho b \sigma u_1u_2 \right]
\;=\; \frac{2(\beta-\alpha)^2}{\sigma^2+b^2 -2\rho b\sigma}.
\end{eqnarray}
Thus, since
\[
T = \frac{1}{\alpha-\beta} \, \log\left(\frac{x_0}{y_0}\right),
\]
the optimal rate is
\begin{equation} \label{Optimal-I}
I = \frac{2(\alpha-\beta) \log\left(\frac{x_0}{y_0}\right)}{\sigma^2+ b^2-2\rho b \sigma}.
\end{equation}
\begin{center}
\begin{figure}
\begin{center}
\includegraphics[width=5.6in]{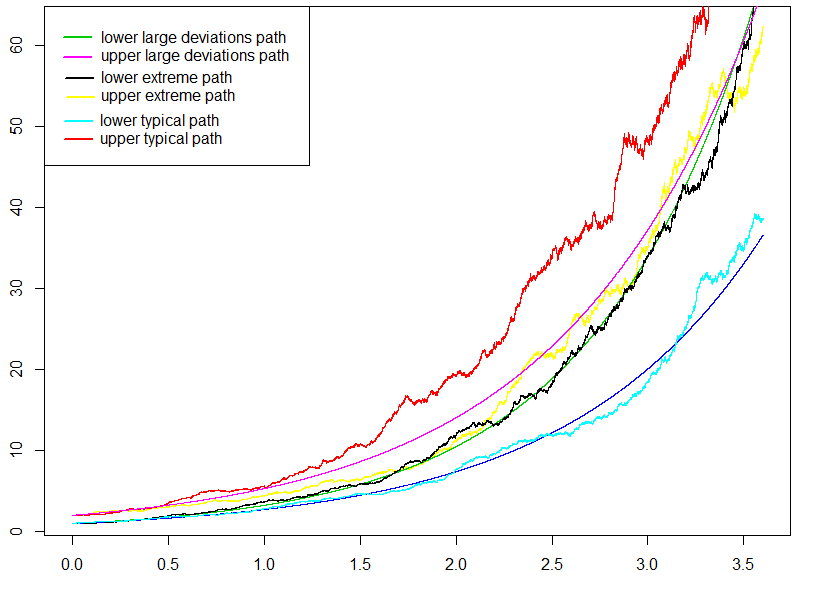}
\caption{Two independent Geometric Brownian Motions. The red and blue paths
are typical upper and lower paths. The yellow and black paths are extreme upper 
and lower paths that meet, following the corresponding magenta and green large 
deviation optimal paths. They were picked by simulating a large number of paths until a 
successful instance, where an upper and a lower path met, was obtained.}
\label{2GBM-optimal-path}
\end{center}
\end{figure}
\end{center}
\subsubsection{Exact analysis for two correlated Geometric Brownian Motions }
An exact analysis is again possible here. Suppose
\[
X_t^\epsilon \;=\; x_0 e^{\left(\alpha - \frac{1}{2} \sigma_\epsilon^2 \right) t + \sigma_\epsilon W_t}, \hspace{0.2in}
Y_t^\epsilon \;=\; y_0 e^{\left(\beta - \frac{1}{2} b_\epsilon^2 \right) t + b_\epsilon B_t},
\]
are two families of Geometric Brownian Motions, indexed by a positive parameter $\epsilon$. 
We will assume that $\sigma_\epsilon = \sigma \sqrt{\epsilon}$
and, similarly,  $b_e = b \sqrt{\epsilon}$. Assuming that $\alpha > \beta$ and $x_0 > y_0$ and that $\{W_t\}$,
$\{B_t\}$. are standard Brownian motions
with correlation $\rho$ as in section \ref{sec:CGBM}, we are interested in obtaining an expression for the probability
\begin{equation} \label{Teps}
\mathbb{P}(T_\epsilon < \infty) \;\; \mbox{ where } T_\epsilon = \inf\{t>0 : Y^\epsilon_t  \geq  X^\epsilon_t\}.
\end{equation}
The condition $Y^\epsilon_t \geq X^\epsilon_t$ is equivalent to
\[
 \left(\alpha -\beta + \frac{1}{2} (b_\epsilon^2 -  \sigma_\epsilon^2) \right) t + \sigma_\epsilon W_t - b_\epsilon B_t \leq \log \frac{y_0}{x_0}.
\]
Set $\log\frac{y_0}{x_0} = - u$, $\gamma_\epsilon := \alpha -\beta + \frac{1}{2} (b_\epsilon^2 -  \sigma_\epsilon^2)$ 
and $\theta_\epsilon :=\sqrt{\sigma^2_\epsilon + b^2_\epsilon-2\rho b_\epsilon \sigma_\epsilon}$. 
If $\{\tilde{W}_t\}$ is standard Brownian motion, then (\ref{Teps}) becomes
\begin{equation} \label{Teps2}
\mathbb{P}(T_\epsilon < \infty) \;=\; 
\mathbb{P}\left( \inf_{t\geq 0} (\gamma_\epsilon t + \theta_\epsilon \tilde{W}_t) < - u \right).
\end{equation}
Since $\alpha > \beta$, when $\epsilon$ is sufficiently small, $\gamma_\epsilon >0$ regardless of the values of 
$\sigma$ and $b$. Therefore (see Cox and Miller, 1977) (\ref{Teps2}) becomes
\[
\mathbb{P}(T_\epsilon < \infty) \;=\; e^{-u \frac{2\gamma_\epsilon}{\theta_\epsilon^2}}
\;=\; e^{\log\frac{y_0}{x_0}\, \frac{ 2(\alpha -\beta) + (b_\epsilon^2 -  \sigma_\epsilon^2)}{\sigma^2_\epsilon + b^2_\epsilon-2\rho b_\epsilon \sigma_\epsilon}}.
\]
It therefore follows that
\[
\lim_{\epsilon \rightarrow 0} \epsilon \log \mathbb{P}(T_\epsilon < \infty) \;=\; \log\frac{y_0}{x_0}\, \lim_{\epsilon \rightarrow 0}
 \frac{ 2(\alpha -\beta) +  (b_\epsilon^2 -  \sigma_\epsilon^2)}{\epsilon^{-1} \left(\sigma^2_\epsilon + b^2_\epsilon-2\rho b_\epsilon \sigma_\epsilon\right)} \;=\;
 \log\frac{y_0}{x_0}\,  \frac{2(\alpha -\beta)}{\sigma^2+ b^2-2\rho b \sigma}.
\]
This result of course agrees with (\ref{Optimal-I}).

\section{Appendix}

\subsection{The paths $x(\cdot)$ and $V(\cdot)$.} \label{sec:app-t2}
Here we refer to part 3 of the proof of Theorem \ref{th:OU-1}. The comparison between the slope of the optimal path
$x(\cdot)$ and $V(\cdot)$ at the intersection point $t$ is given by the following
\begin{proposition}  \label{prop:s2}
\begin{equation} \label{ipc-d}
\sgn( x'(t) - V'(t)) = \left\{ \begin{array}{rcc} -1 & \mbox{if} & t < t_2 \\ 0 & \mbox{if}&  t=t_2 \\ +1 & \mbox{if} & t >t_2 \end{array}
\right.
\end{equation}
where $t_2$ is the unique solution of the equation $\phi_2(t)  \;=\; 2\frac{x_0}{v_0}$
with
\begin{equation} \label{ph2_def}
\phi_2(s) \; :=\; \left( 1 - \frac{\beta}{\mu} \right) e^{(\beta+\mu)s} + \left( 1 + \frac{\beta}{\mu} \right) e^{(\beta-\mu)s}.
\end{equation}
Also,
\begin{equation} \label{t1t2}
t_1 < t^o_V< t_2
\end{equation}
where $t_1$ is defined in (\ref{inequality-condition-1}) and $t^o_V$ in (\ref{eq-phi1}).
\end{proposition}
\begin{proof}
Taking into account (\ref{optimal-path-finite}),
$x'(t)-V'(t) \;=\; \mu \frac{v_0 e^{\beta t} \left(e^{\mu t}+ e^{-\mu t} \right) - 2x_0 }{e^{\mu t} - e^{-\mu t}} -
\beta v_0 e^{\beta t}$ and hence
\begin{equation} \label{Derivative-Condition}
x'(t)-V'(t) <0 \; \Leftrightarrow \; \left( 1 - \frac{\beta}{\mu} \right) e^{(\beta+\mu)t}
+ \left( 1 + \frac{\beta}{\mu} \right) e^{(\beta-\mu)t} < 2 \frac{x_0}{v_0}.
\end{equation}
Defining the function $\phi_2$ via (\ref{ph2_def}) we note that
$\phi'_2(s) = (\mu^2-\beta^2) \frac{2}{\mu}e^{\beta s} \sinh(\mu s) >0$ for all $s\geq 0$ and
$\phi_2(0)=2$. Hence, the equation $\phi_2(s) \;=\; 2\frac{x_0}{v_0}$
has a unique, positive solution, say $t_2$. Since the function $\phi_2(s)$ is continuous and strictly increasing this
establishes (\ref{ipc-d}).

Next we will show that
\begin{equation} \label{inequalityt2o}
t^o_V \;<\; t_2 .
\end{equation}
Indeed, using the definition of $\phi_1$ and $t^o$,
\begin{eqnarray*}
\phi_2(t^o_V) &=& \left( 1 - \frac{\beta}{\mu} \right) e^{(\beta+\mu)t^o_V}
+ \left( 1 + \frac{\beta}{\mu} \right) e^{(\beta-\mu)t^o_V} \\
&=& \phi_1(t^o_V) \;+\;  e^{(\beta-\mu)t^o_V} < \frac{x_0}{v_0} + 1 \;<\; 2 \frac{x_0}{v_0} \;=\; \phi_2(t_2).
\end{eqnarray*}
where we have used the fact that $\beta-\mu <0$ and that $x_0 > v_0$. Then (\ref{inequalityt2o}) follows from the fact
that $\phi_2$ is increasing.

Finally note that
$\phi_2(t_1) = \frac{x_0}{v_0}\left( 1-\frac{\beta}{\mu} + \left(1+\frac{\beta}{\mu}\right) e^{-2\mu t} \right)
< 2 \frac{x_0}{v_0}$
which implies, since $\phi_2$ is strictly increasing, (\ref{t1t2}).
\end{proof}

Define the function $h(s):=x(s)-V(s)$. We have $h(0)=x(0)-V(0)=x_0-v_0 >0$. Also $h(t) =x(t)-V(t)=0$.
We will show that, when $t> t_1$, there are precisely two zeros of the function $h$ on $[0,\infty)$, $t$ and
$\tau(t)$. When $t \in (t_1,t_2)$ $\tau(t) >t$ whereas when $t>t_2$, $\tau(t)<t$. In the special case $t=t_2$,
$\tau(t_2)=t_2$ is the single zero of $h$ at which $h'$ also vanishes.

We have
\begin{equation} \label{hprime}
h'(s) \;=\; \mu c_1 e^{\mu s} - \mu c_2 e^{-\mu s} + \beta v_0 e^{\beta s}.
\end{equation}
The following proposition gives some qualitative properties of this function.
\begin{proposition}
Suppose $t>t_1$. Then there exists $s_1(t)>0$ such that $h'(s)<0$ when $s<s_1(t)$, $h'(s_1)=0$, 
and $h'(s)>0$ when $s>s_1(t)$. Also $\lim_{s\rightarrow \infty} h'(s) = +\infty$ and the following holds:   
There are precisely two values for which the function $h$ vanishes. One is $t$ while the second we 
denote by $\tau(t)$. If $t < s_1(t)$ then $\tau(t)>t$ while if $t>s_1(t)$ then $\tau(t)< t$.
When $t=s_1(t)$ then $t=\tau(t)$ and $h(t)=h'(t)=0$.
\end{proposition}
\begin{proof} First we will show that $h'(0)<0$. Indeed,
\begin{eqnarray*}
h'(0) &=& \mu(c_1-c_2) -\beta v_0 \;=\; \mu\; 
\frac{ 2v_0 e^{\beta t} -x_0(e^{\mu t} + e^{-\mu t})}{e^{\mu t} - e^{-\mu t}} - \beta v_0 \\
&=& \frac{2\mu}{e^{2\mu t} - 1} \, v_0 e^{(\beta+\mu) t} -v_0 \beta - \mu x_0 \, \frac{e^{2\mu t} +1}{e^{2\mu t} - 1} \\
&=& \frac{2\mu}{e^{2\mu t} - 1} \, v_0 
\left(e^{(\beta+\mu) t}-1\right)  -v_0 \beta +v_0 \frac{2\mu}{e^{2\mu t} - 1}
- \mu x_0 \, - x_0 \frac{2\mu}{e^{2\mu t} - 1} \\
&=& \frac{\int_0^t e^{(\beta+\mu)\xi} d\xi }{\int_0^t e^{2\mu \xi} d\xi } \, (\beta+\mu) v_0   -v_0 \beta +v_0 \frac{2\mu}{e^{2\mu t} - 1}
- \mu x_0 \, - x_0 \frac{2\mu}{e^{2\mu t} - 1}
\end{eqnarray*}
The ratio of integrals above is seen to be less than one (since $\beta <\mu$) and hence
\begin{eqnarray*}
h'(0) &\leq &  (\beta+\mu) v_0   -v_0 \beta +v_0 \frac{2\mu}{e^{2\mu t} - 1}
- \mu x_0 \, - x_0 \frac{2\mu}{e^{2\mu t} - 1} \\
&=& (v_0 - x_0) \, \left( \mu + \frac{2\mu}{e^{2\mu t} - 1} \right)  < 0.
\end{eqnarray*}

From (\ref{hprime}) we see that $h'(s) = e^{\mu s} h_1(s)$ with
$h_1(s) := \mu c_1  - \mu c_2 e^{-2\mu s} + \beta v_0 e^{-(\mu-\beta) s}$. 
Clearly $h'(s)$ and $h_1(s)$ have the same sign.
Also, $h_1(0)= h'(0)<0$ and since $c_1 >0$, $c_2>0$, (the first because $t>t_2$) and $\mu>\beta$, 
it follows that $h_1(s)$ is strictly increasing in $s$ and satisfies 
$h_1(s) \uparrow \mu c_1 >0$ as $s\uparrow \infty$. Therefore there exists a unique
$s_1>0$ such that $h_1(s_1)=0$.

We have course $h(t)=0$. Since the value of $s_2$ determined in Proposition \ref{prop:s2} 
depends on $t$ we will use the notation $s_2(t)$. Then,
\begin{itemize}
\item If $t \in (t_1,t_2)$  then, from Proposition \ref{prop:s2}, $h'(t)<0$ which implies, 
in view of the above analysis that $s_2(t) >t$. This in turn means that
$\tau(t)>s_2(t)$ and hence that $t<\tau(t)$.
\item If $t =t_2$ then $h'(t_2)=0$ which implies that $s_2(t_2) =t_2$.
\item If $t >t_2$ then $h'(t)>0$ which implies that $s_2(t) <t$ and hence that $\tau(t)< s_2(t)$. 
Thus in this case $\tau(t) <t$.
\end{itemize}
This concludes the proof of the proposition.
\end{proof}

\subsection{A time-change approach to the Ornstein-Uhlenbeck ruin problem}  \label{sec:time-change}
Consider the two sided problem
\[
dX_t \;=\; \mu X_t dt + \sigma dW_t, \;\;\; X_0=x_0
\]
with an upper boundary given by the curve $U(t):=u_0 e^{\alpha t}$ and a lower boundary given 
by $V(t):=v_0 e^{\beta t}$. We assume that $0<v_0 < x_0 < u_0$ and $0< \beta < \mu < \alpha$. 
We are interested in the hitting time $T = \inf\{t\geq 0: X_T \geq U(T) \mbox{ or } X_T \leq V(T)\}$. 
(Of course, if the set is empty, the hitting time is equal to $+\infty$ corresponding to the case where 
the process never exits from one of the two boundary curves.)
Ignoring the parametrization by $\epsilon$ since here we want to discuss exact results, 
the SDE (\ref{SDEep1})  has the solution
\begin{equation} \label{exact_OU}
X_t \;=\; x_0 e^{\mu t} + \sigma \int_0^t e^{\mu(t-s)} dW_s   
\end{equation}
The condition
\[
V(t) < X_t < U(t)
\]
is equivalent to $e^{-\mu t} V(t) < e^{-\mu t} X_t < e^{-\mu t} U(t)$ or
\begin{equation} \label{double-ineq}
v_0 e^{-(\mu -\beta) t} < x_0 + \sigma \int_0^t e^{-s\mu} dW_s < u_0 e^{(\alpha-\mu) t}.
\end{equation}
The stochastic integral $\xi(t):=\sigma \int_0^t e^{-s\mu} dW_s$ is a Gaussian process 
with independent intervals and variance function
\[
{\sf Var}(\xi(t))= \sigma^2 \int_0^t e^{-2\mu s}ds= \frac{\sigma^2}{2 \mu} \left(1-e^{-2\mu t}\right).
\]
Note that the limit $\lim_{t\rightarrow \infty}{\sf Var}(\xi(t))=\frac{\sigma^2}{2 \mu}$ is finite. 
Consider the time change function $\tau(t)$ defined by
\begin{equation} \label{time-change}
\tau(t) = \frac{\sigma^2}{2 \mu} \left(1-e^{-2\mu t}\right), \hspace{0.3in} t \in [0,\infty)
\end{equation}
The inverse function (which necessarily exists since ${\sf Var}(\xi(t))$ is an increasing function) is
\begin{equation} \label{time-change-2}
t(\tau) = \log\left( 1- \frac{2\mu \tau}{\sigma^2}\right), \hspace{0.3in} \tau \in \left[0,\frac{\sigma^2}{2\mu}\right)
\end{equation}

Applying this change of time to the double inequality (\ref{double-ineq}) we obtain
\[
v_0 e^{-(\mu -\beta) \log\left( 1- \frac{2\mu \tau}{\sigma^2}\right)}
< x_0 + \sigma \int_0^{\log\left( 1- \frac{2\mu \tau}{\sigma^2}\right)} e^{-s\mu} dW_s
< u_0 e^{(\alpha-\mu) \log\left( 1- \frac{2\mu \tau}{\sigma^2}\right)}, \hspace{0.3in} \tau \in \left[0,\frac{\sigma^2}{2\mu}\right).
\]
However, $\tilde{W}_\tau := \sigma \int_0^{\log\left( 1- \frac{2\mu \tau}{\sigma^2}\right)} e^{-s\mu} dW_s$ 
is standard Brownian motion. (It can easily be seen that it is a continuous martingale with quadratic variation 
function $\langle \tilde{W} \rangle_\tau =\tau$.)
Thus we have the equivalent problem
\begin{equation} \label{double-ineq-tc}
v_0 \left( 1- \frac{2\mu \tau}{\sigma^2} \right)^{\frac{\mu-\beta}{2\mu}} \;<\; x_0 \;+\; \tilde{W}_\tau \;<\;
u_0 \left( 1- \frac{2\mu \tau}{\sigma^2} \right)^{-\frac{\alpha-\mu}{2\mu}}\;, 
\hspace{0.3in} \tau \in \left[0,\frac{\sigma^2}{2\mu}\right).
\end{equation}

\begin{center}
\begin{figure}
\begin{center}
\includegraphics[width=3.0in]{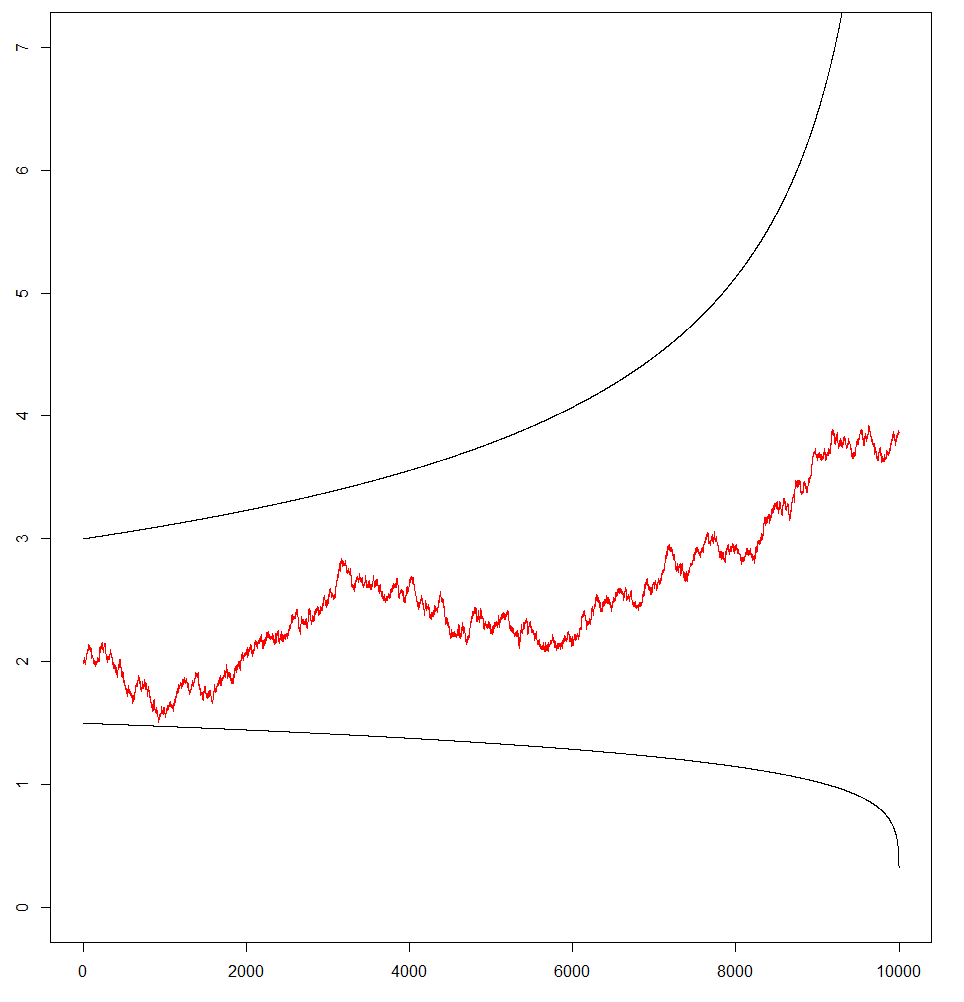}
\caption{Time-change in an Ornstein-Uhlenbeck ruin problem.}
\label{fig:time-change}
\end{center}
\end{figure}
\end{center}

In general, the passage time -- hitting probability problem associated with (\ref{double-ineq-tc}) must be
solved numerically. Of course the time change transformation may have computational advantages. There is a great
deal of work, both theoretical and applied, regarding passage times and hitting probabilities of Brownian motion
with curving boundaries. In the special case where $\alpha = \beta = \mu$ an exact solution exists. In general we
have not been able to obtain closed form expressions even with a single boundary even in the few cases where
exact solutions are known, such as for a parabolic boundary: When $\beta=0$ then the time-changed lower bound
is $v_0\sqrt{ 1- \frac{2\mu \tau}{\sigma^2}}$. While this is a parabolic boundary, the results that have obtained
for this case, Durbin and Williams (1992), Durbin (1985), apply when it acts as an {\em upper} and not a lower 
boundary. The exact solution in the case we are examining is not known, to the best of our knowledge. (See also 
Novikov, 1981, Herrmann and Tanr\'e, 2016, and the references therein.) 

\paragraph{A two--boundary case: $\alpha = \beta = \mu$.} In that case (\ref{double-ineq-tc}) becomes
\[
v_0 - x_0 \; < \; \tilde{W}_\tau \; < \; u_0-x_0,  \hspace{0.2in} \tau \in \left[0,\frac{\sigma^2}{2\mu}\right).
\]
The exact probability of never exiting either boundary, can be obtained from the well known expression for the
density of standard Brownian motion (starting at zero) with absorbing boundaries at $a$, $b$, ($a,b >0$).
If $p(x,t)dx := \mathbb{P}\left(W_t \in (x,x+dx);\; -b < W_s < a,\; 0\leq s \leq t\right)$, then,
(see Cox and Miller, 1977, p. 222)
\begin{eqnarray*}
p(x,t) &=& \sum_{n=1}^\infty \frac{2}{a+b} \sin\left( \frac{n\pi b}{a+b} \right) 
e^{-\lambda_n t} \,\sin\left( n\pi \frac{x+b}{a+b}\right), \\
&&\hspace{0.2in} \mbox{ where } \lambda_n = \frac{1}{2} \frac{n^2 \pi^2}{(a+b)^2}, \;\; n=1,2,\ldots.
\end{eqnarray*}
Then
\[
\mathbb{P}\left(-b < W_s < a,\;\mbox{ for } 0\leq s \leq t \right) = \int_{-b}^a p(x,t) dx,
\]
and in our case $-b = v_0-x_0$, $a= u_0-x_0$, $t= \frac{\sigma^2}{2\mu}$. Hence,
\begin{eqnarray}   \label{series-boundary}
 \mathbb{P}\left(-b < W_s < a,\; 0\leq s \leq \frac{\sigma^2}{2\mu} \right) && \\    \nonumber
&& \hspace{-1.5in}\;=\; \sum_{k=0}^\infty \frac{4}{(2k+1)\pi}
\exp\left({-\frac{(2k+1)^2\pi^2\sigma^2}{2(u_0-v_0)^2\mu}}\right)  
\sin\frac{(2k+1)\pi (x_0-v_0)}{u_0-v_0} .
\end{eqnarray}

\subsection{Exact asymptotics via Gaussian process theory} \label{sec:Gaussian}
The rich and extensive theory of extreme value analysis for Gaussian processes may
be used in some of the problems considered in this paper in order to obtain exact (as opposed to
logarithmic) asymptotics. For a review of results in this area and applications to ruin problems
we refer the reader to the monograph by Piterbarg (1996) and to H\"{u}sler and Piterbarg (1999),  
Hashorva and H\"{u}sler (2000), and D\c{e}bicki, Hashorva and Li (2015). 
The ruin problems discussed in Section 2 involve Ornstein-Uhlenbeck processes which are Gaussian
but are not stationary or self-similar, nor do they have stationary increments. However, with appropriate 
scaling one may obtain a Gaussian process with a single point of maximum variance and the technique 
described in Piterbarg (1996) p. 19 and H\"{u}sler and Hashorva (2000) may be applied.
 
We illustrate this briefly in one of the problems of Section \ref{sc:WF-FHP}. Write the solution of the SDE 
(\ref{SDEep1}) as 
\[
X_t^\epsilon = x_0 e^{\mu t} + \epsilon^{1/2} \sigma \int_0^t e^{\mu(t-u)}dW_u.
\]
Then
\[
\bP\Big(\sup_{t\in [0,T]} X_t^\epsilon > u_0e^{\alpha t} \Big) \;=\; 
\bP \Big(\sup_{t\in [0,T]} Z_t >  \epsilon^{-1/2} \Big)
\]
where $\{Z_t;t\geq 0\}$ is the centered Gaussian process defined by
\[
Z_t \; := \; \frac{\sigma}{u_0 e^{\alpha t} - x_0 e^{\mu t}}  \int_0^t e^{\mu(t-u)}dW_u .
\]
(Thus we pass from low noise asymptotics to high threshold asymptotics.)
The covariance function of the process $\{Z_t; t\geq 0\}$ can be seen by a simple 
comptutation to be
\[
R(s,t) \;=\; \frac{\sigma^2}{2\mu} \, 
\frac{  e^{\mu(t+s)} -e^{\mu(t-s)}}
{\Big( u_0 e^{\alpha s} - x_0 e^{\mu s} \Big)\Big( u_0 e^{\alpha t} - x_0 e^{\mu t} \Big)},
\;\;\; 0 \leq s \leq t \leq T.
\]
and in particular the variance function $S^2(t):={\sf Var}(Z_t) = R(t,t)$ is given by 
\begin{equation}  \label{max_var}
S^2(t) \; =\;  \frac{\sigma^2}{2\mu} \, 
\frac{  e^{2 \mu t} -1}{\Big( u_0 e^{\alpha t} - x_0 e^{\mu t} \Big)^2} .
\end{equation}
Denote by $\mathcal{H}_{x_0,U(t)}^1$ the set of absolutely continuous functions
$x:[0,t]\rightarrow \bR$  with square integrable functions, such that $x(0)=x_0$ and 
$x(t)=U(t):=u_0e^{\alpha t}$ (c.f. the definition of $\mathcal{H}^1_{x_0,V(t)}$ in
the proof of Part 1 of Theorem \ref{th:OU-1}) and set
\[
J_*(t) \; := \; \inf\{J(x;t): x \in \mathcal{H}_{x_0,U(t)}^1 \}
\]
for the infimum of the action functional for paths in $\mathcal{H}_{x_0,U(t)}^1$.
An analysis paralleling step for step the derivation of (\ref{rate2}) shows that
the minimum value of the action functional $J(x;t)$ is
\begin{equation} \label{rate2*}
J_*(t) \;=\; \frac{\mu}{\sigma^2} \frac{\Big( u_0 e^{\alpha t} - x_0 e^{\mu t} \Big)^2}
{e^{2\mu t} -1}.
\end{equation}
Furthermore, if we allow $t$ to vary in $[0,T]$, $J_*(t)$ has a strict global minimum at 
$t_0^U$ which is the unique positive solution of 
equation (\ref{topt-u}) (c.f. (\ref{th1-2}) and (\ref{OU-IU})). Assuming that $T> t_0^U$, 
the rate function is $I_U(T)=J_*(t_0^U)$.
However, by comparing (\ref{max_var}) and (\ref{rate2*}), we see that 
$S^2(t) = (2J_*(t))^{-1}$ and hence $t_0^U$ is a unique point of maximum variance
for the Gaussian process $\{Z_t;t\geq 0\}$ with corresponding value 
\begin{equation}   \label{maximum_variance}
S^2(t_0^U)= \frac{1}{2J_*(t_0^U)} = \frac{1}{2 I_U(T)}. 
\end{equation} 
Define now the process $\{\widetilde{Z}_t; t\geq 0\}$ via 
\[
\widetilde{Z}_t = \sqrt{2 I_U(T)}\, Z_t.
\]
This is a centered Gaussian process with a.s. continuous paths and variance function 
\[
\widetilde{S}^2(t):={\sf Var}(\widetilde{Z}_t) = S^2(t)  2 I_U(T) \leq 1
\]
with equality holding only at $t=t_0^U$. We claim that his process satifies conditions
E.1, E.2, and E.3 in Piterbarg (1996), p.19. Since $(S^2)'(t_0^U)=0$ we can see 
by a Taylor expansion around $t_0^U$ that
\begin{equation}  \label{Scondition}
\widetilde{S}(t) \;=\; 1 - \kappa_1 \big(t-t_0^U \big)^2+ o\big((t-t_0^U)^2\big) 
\end{equation} 
where 
\begin{equation} \label{kappa1}
\kappa_1 = 2 I_U(t_0^U) \cdot (S^2)''(t_0^U) .
\end{equation} 
(This is condition E.1.) The covariance function of $\{\widetilde{Z}_t\}$ is 
$\widetilde{R}(s,t) := \frac{I_U(T)}{2} R(s,t)$
and (by a simple but tedious calculation) it can be shown that it is 
{\em locally stationary} at $t_0^U$, i.e.\ that it satisfies the condition
\begin{equation} \label{Rcondition}
\widetilde{R}(s,t) \;=\; 1- \kappa_2 |t-s| (1+o(1)),\;\;\; \mbox{ as $|t-t_0^U|\rightarrow 0$, 
$|s-t_0^U|\rightarrow 0$.}
\end{equation}
where 
\begin{equation} \label{kappa2}
\kappa_2 = \frac{\mu}{e^{2\mu t_0^U} -1}.
\end{equation}
(This is condition E.2. See also Hashorva and H\"{u}sler, 2000.) 
Finally one can also check by a straighforward computation that 
$\bE \big( \widetilde{Z}(t) - \widetilde{Z}(s)\big)^2 \leq G |t-s|$
for some $G>0$, i.e. that the regularity condition E.3 in Piterbarg (1996) is satisfied. 
Therefore, applying Theorem D3 of Piterbarg (1996) (p.19) (choosing appropriately
$\alpha=1$, $\beta=2$ in that theorem). 
\begin{eqnarray}  \nonumber
\bP(\sup_{t\in [0,T]} X_t^\epsilon > u_0e^{\alpha t} ) &=& 
\bP \Big(\sup_{t\in [0,T]} Z_t >  \epsilon^{-1/2} \Big) \;=\; 
\bP\Big(\max_{t\in [0,T]} \widetilde{Z}_t > \epsilon^{-1/2} (2I_U(T))^{1/2}\Big)  \\  
\label{Piterbarg} 
&=& \frac{\kappa_2 \mathcal{H}_1 \Gamma(1/2)}{ \kappa_1^{1/2}} \, 
\frac{\epsilon^{1/2}}{2 I_U(T)}\, \Psi\big((2 I_U(T)/\epsilon)^{1/2}\big) \, \big(1+o(1)\big).
\end{eqnarray}
In the above, $\Psi(u):= \bP(\mathcal{N} >u)$ where $\mathcal{N}$ is a standard normal 
r.v. and $\mathcal{H}_1$ is the value of Pickands'  constant 
$\mathcal{H}_\alpha$ for $\alpha=1$ which is equal to 1. 

Taking into account the fact that $\lim_{u\rightarrow \infty} \sqrt{2\pi}\,u\, e^{u^2/2} \Psi(u) =1$, 
the {\em exact asymptotics} for the probability of hitting the upper curve {\em when the maximum 
variance point $t_0^U$ is an interior point of $[0,T]$} are expressed as
\begin{equation} \label{exact_asymptotics}
\bP(\sup_{t\in [0,T]} X_t^\epsilon > u_0e^{\alpha t} ) \;=\; 
\frac{\kappa_2 }{\sqrt{2 \kappa_1}} \, e^{-\frac{1}{\epsilon} \, I_U(T)} (1 + o(1)).
\end{equation}
This agrees with the logarithmic asymptotics in (\ref{th1-2}) and (\ref{OU-IU}) for the case 
where $t_0^U <T$. 

The case $t_0^U \geq T$ corresponds to the case where the variance of the Gaussian process 
$\{Z_t;t\geq 0\}$ has a maximum at the boundary point of the interval, $T$ and the exact 
asymptotics are obtained by multiplying (\ref{exact_asymptotics}) by a factor of $1/2$. 
(See Konstant and Piterbarg, 1993, Theorem 2.2.)



\vspace{0.4in}

 \begin{center}

\begin{figure}[h!]
\includegraphics[width=5.0in]{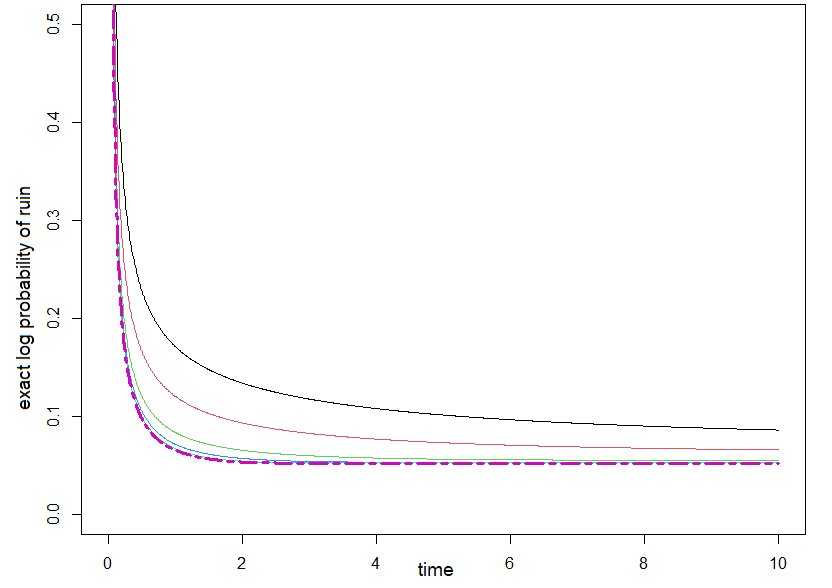}
\caption{Negative logarithm of hitting probability for $\sigma = 3,\, 2,\, 1,\, 0.5,\, 0.2$ 
and comparison with the Wentzell–Freidlin low variance limit.}
\label{fig:GBM-3}
\end{figure}

%
%
\begin{figure}
\includegraphics[width=4.0in]{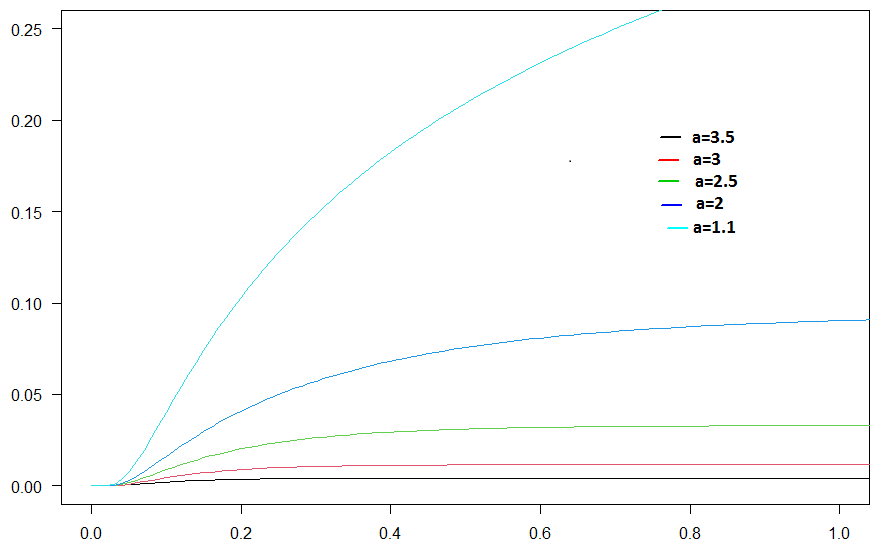}
\caption{Probability of hitting the upper boundary as a function of time horizon based on the exact solution
(\ref{the-exact-solution}). Here $\sigma=0.5$, $x_0=1$, $u_0=1.3$, $\mu=1$.
The function is plotted for $\alpha=1.1$, 2, 2.5, 3, 3.5.}
\end{figure}

\begin{figure}
\includegraphics[width=4.0in]{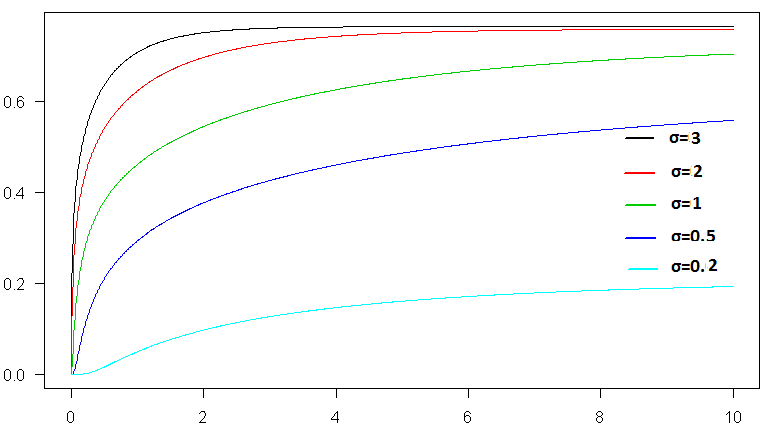}
\caption{Probability of hitting the upper boundary as a function of time horizon based on the exact solution
(\ref{the-exact-solution}). Here  $x_0=1$, $u_0=1.3$, $\mu=1$ $\alpha=1.1$.
The function is plotted for $\sigma=0.2$, 0.5, 1, 2, 3.}
\end{figure}

\begin{figure}
\includegraphics[width=4.0in]{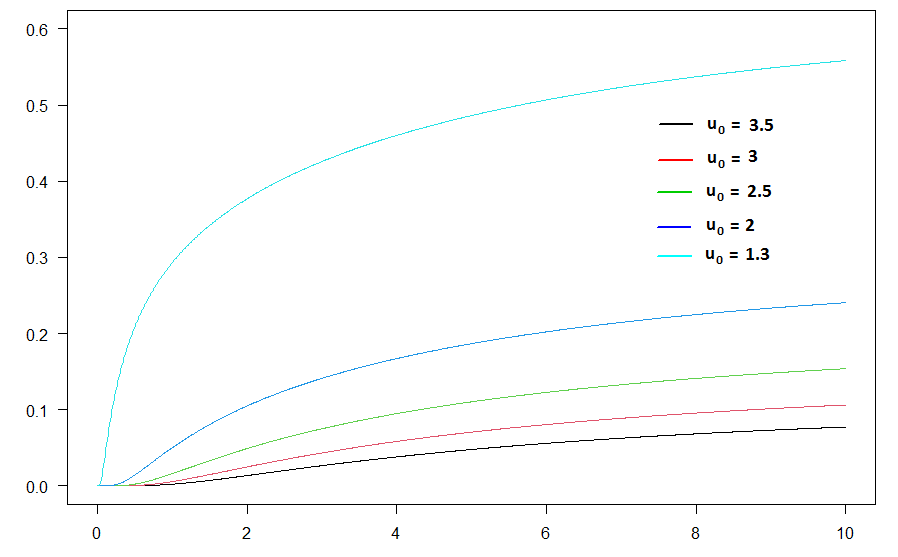}
\caption{Probability of hitting the upper boundary as a function of time horizon based on the exact solution
(\ref{the-exact-solution}). Here  $x_0=1$, $\alpha=1.1$, $\mu=1$ $\sigma=0.5$.
The function is plotted for $u_0=1.3$, 2, 2.5, 3, 3.5.}
\end{figure}

\begin{figure}
\includegraphics[width=4.5in]{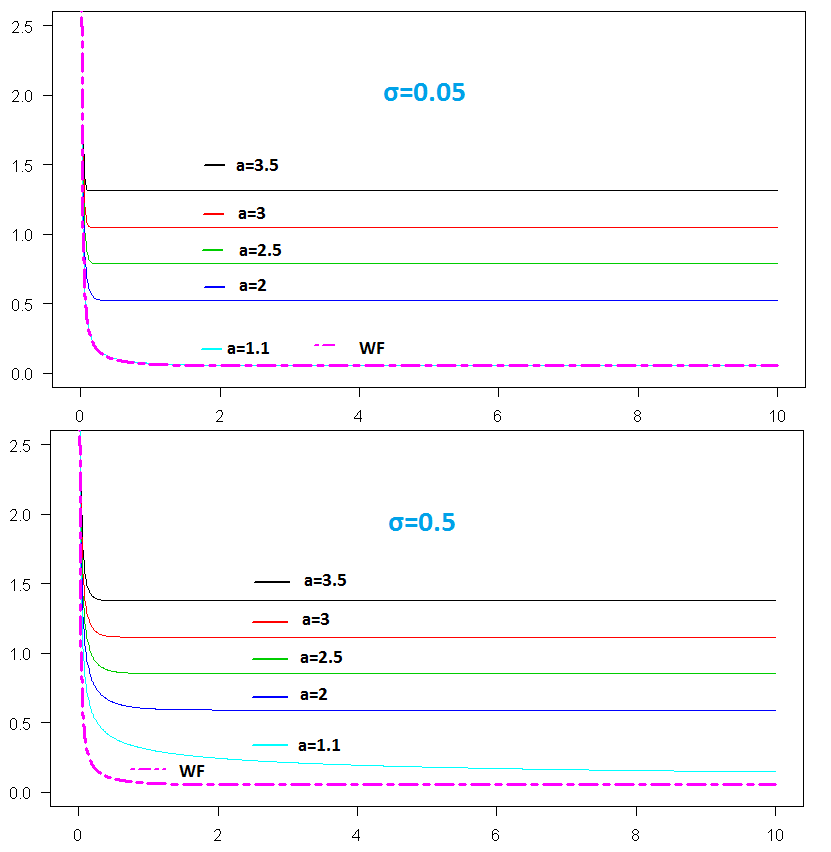}
\caption{ $-log$ Probability of hitting the upper boundary based on the exact solution (\ref{the-exact-solution}).
Here  $x_0=1$, $u_0=1.3$, $\mu=1$. The upper graph was obtained for $\sigma=0.05$ while the lower for
$\sigma=0.5$. The magenta dotted line gives the value of (the exponent of) the Wentzell-Freidlin approximation.}
\end{figure}

\begin{figure}
\includegraphics[width=4.5in]{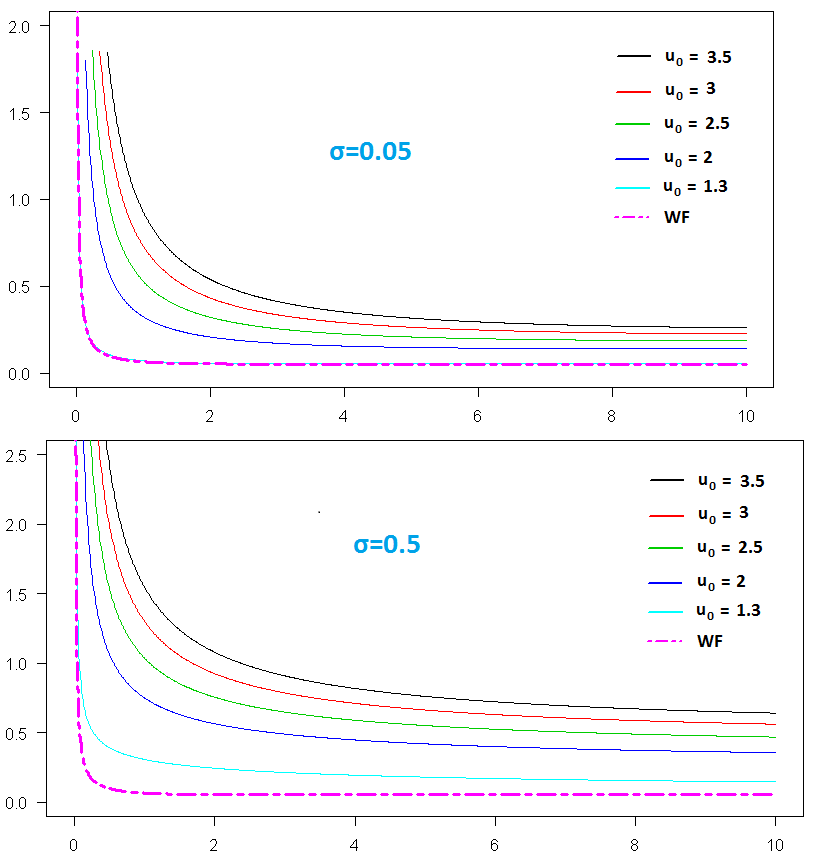}
\caption{ $-log$ Probability of hitting the upper boundary based on the exact solution (\ref{the-exact-solution}).
Here  $x_0=1$,  $\mu=1$, $\alpha=1.3$. The upper graph was obtained for $\sigma=0.05$ while the lower
for $\sigma=0.5$. The magenta dotted line gives the value of (the exponent of) the Wentzell-Freidlin approximation.}

\end{figure}
\begin{figure}
\includegraphics[width=4.5in]{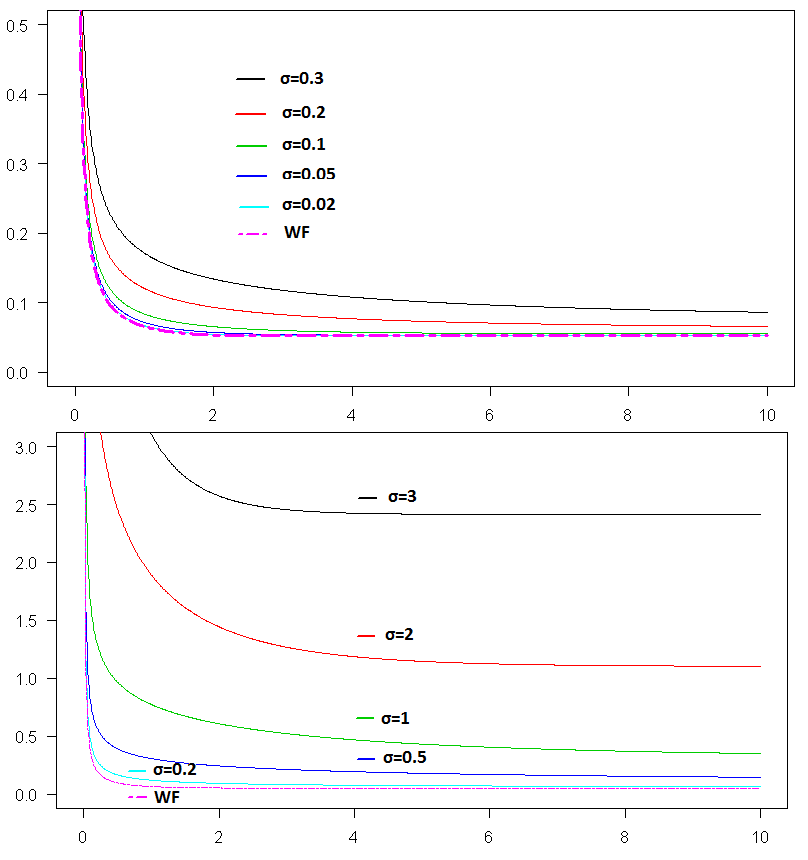}
\caption{ $-log$ Probability of hitting the upper boundary based on the exact solution (\ref{the-exact-solution}).
Here  $x_0=1$, $u_0=1.3$,  $\mu=1$, $\alpha=1.1$. The magenta dotted line gives the value of (the exponent of)
the Wentzell-Freidlin approximation.}
\end{figure}
\end{center}


\begin{thebibliography}{99}

\bibitem{AA} Asmussen, S. and M. Albrecher. {\em Ruin Probabilities}, 
Second Edition. World Scientific. 2010.

\bibitem{Azencott1}
Azencott, R. ``Petites perturbations al\'eatoires des syst\`emes dynamiques: 
d\'eveloppements asymptotiques.'' {\em Bulletin des sciences math\'ematiques}, 
Vol 109, Num 3, pp 253-308, 1985.

\bibitem{Azencott2}
Azencott R. {\em Grandes deviations et applications}, in Ecole d'\'et\'e
de Probabilit\'e de St. Four VIII-1978, Lect. Notes
Math. 774, Springer, 1-176, 1980.

\bibitem{Baldi} Baldi, P. ``Exact Asymptotics for the Probability of Exit 
from a Domain and Applications to Simulation.''
{\em The Annals of Probability}, 23, 4. 1644-1670, 1995.

\bibitem{BaldiCaramellino}
Baldi, P. and L. Caramellino. ``General Freidlin–Wentzell Large Deviations and 
positive diffusions.'' {\em Statistics and Probability Letters,} 81, 1218-1229, 2011.

\bibitem{Bisewski}
Bisewski, K., K. D\c{e}bicki, and N. Kriukov. 
``Simultaneous ruin probability for multivariate Gaussian risk model.''
{\em Stochastic Processes and their Applications}, 160, 386--408, 2023.

\bibitem{BobrovskyZeitouni}
Bobrovsky, B-Z. and O. Zeitouni. ``Some results on the problem of exit from 
a domain''. {\em Stochastic Processes and their Applications},  41, 241-256, 1992. 


\bibitem{B-M}
Brechtken-Manderscheid, U. {\em Introduction to the Calculus of 
Variations}. Chapman \& Hall. 1994.

\bibitem{Cottrell} Cottrell, M., J-C. Fort, and G. 
Malgouyres. ``Large Deviations and Rare Events in the Study of 
Stochastic Algorithms.'' {\em IEEE Transactions on Automatic Control,} 
AC-28, 9, 907-920, 1983.

\bibitem{CoxMiller}
Cox, D.R. and H.D. Miller. {\em The Theory of Stochastic Processes.} 
Chapman and Hall/CRC, 1977.

\bibitem{Day}
Day, M.V.  ``On the asymptotic relation between equilibrium
density and exit measure in the exit problem,'' {\em Stochastics} 
12:3-4, 303-330, 1984.

\bibitem{Debicki0}
D\c{e}bicki K. ``Ruin probability for Gaussian integrated processes.''
{\em Stochastic Processes and their Applications}, 98, 151--174, 2002.

\bibitem{Debicki}
D\c{e}bicki, K., E. Hashorva, and L. Ji. ``Gaussian risk models with financial constraints.''
{\em Scandinavian Actuarial Journal}, Vol. 2015, No. 6, 469–481, 2015.

\bibitem{Debicki3}
D\c{e}bicki, K., E. Hashorva, and L. Wang.
``Extremes of vector-valued Gaussian processes.'' {\em Stochastic Processes 
and their Applications},  130, 5802--5837, 2020.

\bibitem{DZ} Dembo A. and O. Zeitouni. {\em Large Deviations Techniques 
and Applications.} Second Edition. Springer 2010.

\bibitem{Dieker}
Dieker A.B. ``Extremes of Gaussian processes over an infinite horizon.''
{\em Stochastic Processes and their Applications,} {\bf 115}, 207–248, 2005.



\bibitem{DupuisKushner}
Dupuis P. and H.J. Kushner. ``Large deviations estimates for systems with 
small noise effects, and applications to stochastic systems theory''. {\em SIAM J. 
Control and Optimization} 24, 5, 979-1008, 1986.

\bibitem{Durbin}
Durbin J. ``The First-Passage Density of a Continuous Gaussian Process 
to a General Boundary.'' {\em Journal of Applied Probability,} 22, 1, 
99-122, 1985.

\bibitem{DurbinWilliams}
Durbin J. and D. Williams. ``The First-Passage Density of the Brownian 
Motion Process to a Curved Boundary.'' {\em Journal of Applied 
Probability,} 29, 2, 291-304, 1992.

\bibitem{Fleming}
Fleming, W.H. ``Exit probabilities and optimal stochastic control.'' 
{\em Appl.Math. Optim.} 4, 329-346. 1978.


\bibitem{FW}
Freidlin, M.I. and  A.D. Wentzell, {\em Random Perturbations of 
Dynamical Systems}, 3rd edition, Springer Science \& Business Media, 2012.

\bibitem{FKW} Freidlin, M., L. Koralov,  and A. Wentzell, ``On the Behavior 
of Diffusion Processes with Traps,'' {\em The Annals of Probability}, 45(5), 
3202--3222, 2017.

\bibitem{GS2}
Gerber, H.U. and E.S.W. Shiu, ``Geometric Brownian Motion Models for Assets 
and Liabilities: From Pension Funding to Optimal Dividends,'' {\em North 
American Actuarial Journal}, 7(3), 37-51, 2003.

\bibitem{Harrison}
Harrison, M.J. ``Ruin Problems with Compounding Assets,'' {\em Stochastic 
Processes and their Applications,} 5, 67-79, 1977.

\bibitem{HashorvaHusler}
Hashorva, E. and J. H\"{u}sler. ``Extremes of Gaussian Processes with Maximal
Variance near the Boundary Points.'' {\em Methodology and Computing in Applied 
Probability}, {\bf 2}:3, 255--269, 2000.

\bibitem{HeHu}
He, X.  and Y. Hu. ``Ruin Probability for the Integrated Gaussian Process 
With Force of Interest,'' {\em Appl. Prob.} 44, 685--694, 2007.

\bibitem{HerrmannTarne}
Herrmann, S. and E. Tanr\'e. ``The first-passage time of the brownian motion
to a curved boundary: an algorithmic approach,''  {\em SIAM J. Sci. Comput.} 
38, 1, pp. A196–A215, 2016.

\bibitem{Hill}
Hill K., J. Zanetell, and J. A. Gemmer.
``Most probable transition paths in piecewise-smooth stochastic differential equations,''
{\em Physica D: Nonlinear Phenomena,} 439, article 133424, 2022.

\bibitem{Husler}
H\"usler, J. ``Extreme values and high boundary crossings of locally stationary 
Gaussian processes.'' {\em The Annals of Probability}, {\bf 18}, 3, 1141--1158, 1990. 

\bibitem{HuslerPiterbarg2}
H\"{u}sler, J. and V.I. Piterbarg. ``Extremes of a certain class of Gaussian processes.''
{\em Stochastic Processes and their Applications} 83, 257-271, 1999.

\bibitem{HuslerPiterbarg3}
H\"{u}sler, J. and V.I.  Piterbarg.
``On the ruin probability for physical fractional Brownian motion.''
{\em Stochastic Processes and their Applications}, 113, 315--332, 2004.

\bibitem{HuslerPiterbarg}
H\"{u}sler, J. and V.I. Piterbarg.
``A limit theorem for the time of ruin in a Gaussian ruin problem.''
{\em Stochastic Processes and their Applications} 118, 2014–2021, 2008.

\bibitem{Kifer}
Kifer, Ju. I. ``On small random perturbations of some smooth dynamical systems,'' 
{\em Math. USSR Izv.} 8, 1083-1107, 1974.

\bibitem{Kobelkov}
Kobelkov, S.G. ``The ruin problem for the stationary gaussian process.''
{\em Theory Probab. Appl.  SIAM} 
Vol. 49, No. 1, pp. 155–163, 2005 .

\bibitem{PiterbargKonstant}
Konstant, D.G. and V.I. Piterbarg. ``Extreme Values of the Cyclostationary Gaussian 
Random Process.'' {\em Journal of Applied Probability,} {\bf 30}, 1, pp. 82--97, 1993.

\bibitem{Kushner}
Kushner, H. J. ``Robustness and Approximation of Escape Times and Large 
Deviations Estimates for Systems with Small Noise Effects,'' {\em SIAM J. Appl. 
Math.} 44, 1, 160-182, 1984. 

\bibitem{Ludwig}
Ludwig, D. ``Persistence of Dynamical Systems under Random Perturbations,''
{\em SIAM Review,} 17, 4, 605-640, 1975.

\bibitem{MaierStein}
Maier, R.S. and D.L. Stein. ``Limiting Exit Location Distributions in the Stochastic 
Exit Problem,'' {\em SIAM J. Appl. Math.} 57, 3, 752-790, 1997.

\bibitem{Matkowsky-Schuss} Matkowsky, R. J. and Z. Schuss. ``The Exit Problem for 
Randomly Perturbed Dynamical Systems,'' {\em SIAM J. Appl. Math.} 33, 2, 365-382, 1977.

\bibitem{Novikov}
Novikov, A.A. ``On estimates and the asymptotic behavior of nonexit probabilities of a
wiener process to a moving boundary,''  {\em USSR Sb.} 38,  495-505, 1981.


\bibitem{PiterbargAMS}
Piterbarg, V.I. {\em Asymptotic Methods in the Theory of Gaussian Processes and Fields.}
American Mathematical Society, Providence RI, 1996. 

\bibitem{Pham}
Pham H. ``Some Applications and Methods of Large Deviations
in Finance and Insurance,'' in {\em Paris-Princeton Lectures on 
Mathematical Finance 2004},  Carmona, R. et al. Eds, pp. 191-244,
Springer, 2007.


\bibitem{Priouret}
Priouret P. {\em Remarques sur les petites perturbations de syst\`emes
dynamiques}, Seminaire de Probabilit\'e XVI, Lect. Notes
Math. 921, Springer, 184-200, 1982.

\bibitem{Robertson} 
Robertson, S. ``Sample path Large Deviations and optimal importance
sampling for stochastic volatility models.'' {\em Stochastic Processes and their
Applications,} 120, 66-83, 2010.

\bibitem{Schilder} Schilder, M. ``Asymptotic formulas for Wiener integrals.'' 
{\em Trans. Amer. Math. Soc.} 125, 63–85, 1966.

\bibitem{Simonian} Simonian, A. ``Asymptotic Distribution of Exit Times 
for Small-Noise Diffusion.'' {\em SIAM J. Appl. Math.} 55, 3, 809-826. 1995.



\bibitem{Williams} Williams, R. G. ``The Problem of Stochastic Exit.'' 
{\em SIAM J. Appl. Math.} 40, 2, 208-223. 1981.
 



\end{thebibliography}
\end{document}